\documentclass[10pt,reqno]{amsart}
\usepackage{latexsym,amsmath}
\usepackage{enumerate}
\usepackage{amsfonts}
\usepackage{amssymb}
\usepackage{latexsym}

\usepackage[T1]{fontenc}
\usepackage{fourier}
\usepackage{bbm}

\usepackage[dvips]{epsfig}
\graphicspath{{./Fig_eps/}{./Eps/}}
\DeclareGraphicsExtensions{.ps,.eps}
\usepackage{color}
\usepackage{fullpage}

\newcommand\reroot{\uparrow}
\newcommand\marg{\downarrow}
\newcommand\nix{\,\cdot\,}
\newcommand\vV{\vec V}
\newcommand\vW{\vec W}
\newcommand\vS{\vec S}

\newcommand\ind{\mathrm{ind}}
\newcommand\KL[2]{D\bc{{{#1}\|{#2}}}}

\newcommand\atom{\delta}

\newcommand\thet{\vartheta}

\newcommand{\beq}{\begin{equation}} \newcommand{\eeq}{\end{equation}}

\newcommand\ism{\cong}

\newcommand\G{\vec G}
\newcommand\T{\vec T}

\numberwithin{equation}{section}

\newcommand\bem{\bf\em}
\newcommand\bemph[1]{{\bf\em #1}}
\def\vec#1{\mathchoice{\mbox{\boldmath$\displaystyle#1$}}
{\mbox{\boldmath$\textstyle#1$}}
{\mbox{\boldmath$\scriptstyle#1$}}
{\mbox{\boldmath$\scriptscriptstyle#1$}}}

\DeclareMathOperator{\pr}{\mathrm P}

\newcommand\SIGMA{\vec\sigma}

\newcommand\TAU{\vec\tau}

\newtheorem{definition}{Definition}[section]

\newtheorem{example}[definition]{Example}
\newtheorem{remark}[definition]{Remark}
\newtheorem{theorem}[definition]{Theorem}
\newtheorem{lemma}[definition]{Lemma}
\newtheorem{proposition}[definition]{Proposition}
\newtheorem{corollary}[definition]{Corollary}

\newcommand\dist{\mbox{dist}}

\newcommand\PHI{\vec\Phi}

\newcommand\fG{\mathfrak{G}}
\newcommand\fT{\mathfrak{T}}

\newcommand\cA{\mathcal{A}}
\newcommand\cB{\mathcal{B}}
\newcommand\cC{\mathcal{C}}

\newcommand\cF{\mathcal{F}}
\newcommand\cG{\mathcal{G}}
\newcommand\cE{\mathcal{E}}

\newcommand\cS{\mathcal{S}}
\newcommand\cT{\mathcal{T}}

\newcommand\cL{\mathcal{L}}
\newcommand\cM{\mathcal{M}}

\newcommand\cP{\mathcal{P}}
\newcommand\cX{\mathcal{X}}

\newcommand\cV{\mathcal{V}}

\newcommand\cZ{\mathcal{Z}}

\def\cC{{\mathcal C}}
\def\cE{{\mathcal E}}

\newcommand\eps{\varepsilon}

\newcommand\Var{\mathrm{Var}}
\newcommand\Erw{\mathrm{E}}

\newcommand{\vecone}{\vec{1}}

\newcommand{\Be}{{\rm Be}}

\newcommand\TV[1]{\left\|{#1}\right\|_{\mathrm{TV}}}

\newcommand{\bink}[2] {{\binom{#1}{#2}}}

\newcommand\bc[1]{\left({#1}\right)}
\newcommand\cbc[1]{\left\{{#1}\right\}}
\newcommand\bcfr[2]{\bc{\frac{#1}{#2}}}
\newcommand{\bck}[1]{\left\langle{#1}\right\rangle}
\newcommand\brk[1]{\left\lbrack{#1}\right\rbrack}

\newcommand\abs[1]{\left|{#1}\right|}

\newcommand\RR{\mathbb{R}}

\newcommand{\stacksign}[2]{{\stackrel{\mbox{\scriptsize #1}}{#2}}}
\newcommand{\tensor}{\otimes}

\newcommand{\Erdos}{Erd\H{o}s}
\newcommand{\Renyi}{R\'enyi}
\newcommand{\Lovasz}{Lov\'asz}

\newcommand{\Szemeredi}{Szemer\'edi}

\newcommand\Lem{Lemma}
\newcommand\Prop{Proposition}
\newcommand\Thm{Theorem}
\newcommand\Def{Definition}
\newcommand\Cor{Corollary}
\newcommand\Sec{Section}
\newcommand\Chap{Chapter}

\begin{document}

\title{Harnessing the Bethe free energy$^{*}$}

\author[Bapst and Coja-Oghlan]{Victor Bapst$^{**}$, Amin Coja-Oghlan$^{**}$}
\thanks{$^{*}$
A preliminary version~\cite{OldVersion} of this paper, presented by the first author at RANDOM 2015 and by the seocnd author at the RS\&A 2015 conference,
contained a critical technical error that affected its main results.
This present version is based on similar key insights but the main results are different from the ones stated in~\cite{OldVersion}.
}
\thanks{$^{**}$The research leading to these results has received funding from the European Research Council under the European Union's Seventh 
Framework Programme (FP/2007-2013) / ERC Grant Agreement n.\ 278857--PTCC}

\address{Amin Coja-Oghlan, {\tt acoghlan@math.uni-frankfurt.de}, Goethe University, Mathematics Institute, 10 Robert Mayer St, Frankfurt 60325, Germany.}

\address{Victor Bapst, {\tt bapst@math.uni-frankfurt.de}, Goethe University, Mathematics Institute, 10 Robert Mayer St, Frankfurt 60325, Germany.}

\begin{abstract}
\noindent
A wide class of problems in combinatorics, computer science and physics can be described along the following lines.
There are a large number of variables ranging over a finite domain that interact through constraints that each bind a few variables
and either encourage or discourage certain value combinations.
Examples include the $k$-SAT problem or the Ising model.
Such models naturally induce a Gibbs measure on the set of assignments, which is characterised by its partition function.
The present paper deals with the partition function of problems where the interactions between variables and constraints are induced by a sparse random (hyper)graph.
According to physics predictions, a generic recipe called the ``replica symmetric cavity method'' yields the correct
value of the partition function if the underlying model enjoys certain properties [Krzkala et al., PNAS 2007].
Guided by this conjecture, we prove general sufficient conditions for the success of the cavity method.
The proofs are based on a ``regularity lemma'' for probability measures on sets of the form $\Omega^n$ for a finite $\Omega$
and a large $n$ that may be of independent interest.

\bigskip
\noindent
\emph{Mathematics Subject Classification:} 05C80, 82B44
\end{abstract}

\maketitle

\section{Introduction}\label{Sec_intro}

\noindent
Despite their simplicity, or perhaps because thereof, the first and the second moment method are the most widely used techniques in probabilistic combinatorics.
\Erdos\ employed the first moment method famously to lower-bound the Ramsey number as well as to establish the existence of graphs of high girth
and high chromatic number~\cite{ErdosRamsey,Erdos}.
Even a half-century on, deterministic constructions cannot hold a candle to these probabilistic results~\cite{BRSW,Nesetril}.
Moreover, the second moment method has been used
to count prime factors~\cite{Turan} and Hamilton cycles~\cite{RobinsonWormald}
as well as to determine the two possible values of the chromatic number of a sparse random graph~\cite{AchNaor}.

Yet there are quite a few problems for which the standard first and the second moment methods are too simplistic.
The {\em random $k$-SAT model} is a case in point.
There are $n$ Boolean variables $x_1,\ldots,x_n$ and $m$ clauses $a_1,\ldots,a_m$, where $m=\lceil\alpha n\rceil$ for some fixed $\alpha>0$.
Each clause binds $k$ variables, which are chosen independently and uniformly, and
discourages them from taking precisely one of the $2^k$ possible truth value combinations.
The forbidden combination is chosen uniformly and independently for each clause.

The random $k$-SAT instance $\PHI=\PHI_k(n,m)$ gives rise to a probability measure on the set $\{0,1\}^n$ of all Boolean assignments naturally.
Indeed, for a given parameter $\beta\geq0$ the  {\em Gibbs measure} $\mu_{\PHI,\beta}$ is defined by letting
	\begin{align}\label{eqkSAT1}
	\mu_{\PHI,\beta}(\sigma)&=\frac1{Z_\beta(\PHI)}\prod_{i=1}^m\exp(-\beta\vecone\cbc{\sigma\mbox{ violates }a_i})
		\qquad\mbox{for every assignment $\sigma\in\{0,1\}^n$, where}\\
	Z_\beta(\PHI)&=\sum_{\sigma\in\cbc{0,1}^n}\prod_{i=1}^m\exp(-\beta\vecone\cbc{\sigma\mbox{ violates }a_i})\label{eqkSAT2}
	\end{align}
is called the {\em partition function}.
Thus, the Gibbs measure weighs assignments according to the number of clauses that they violate.
In effect, by tuning $\beta$ we can interpolate between just the uniform distribution on $\{0,1\}^n$ ($\beta=0$) and a measure
that strongly favours satisfying assignments ($\beta\to\infty$).
Hence, if we think of $\PHI$ as inducing a ``height function'' $\sigma\mapsto\#\{\mbox{clauses of $\PHI$ violated by }\sigma\}$ on the set of assignments,
then varying $\beta$ allows us to explore the resulting landscape.
Apart from its intrinsic combinatorial interest,
the shape of the height function, the so-called ``Hamiltonian'', governs the performance of  algorithms such as the Metropolis process or Simulated Annealing.

To understand the Gibbs measure it is key to get a handle on the partition function $Z_\beta(\PHI)$.
Of course, the default approach to this kind of problem would be to apply the first and second moment methods.
However, upon closer inspection it emerges that 
$Z_\beta(\PHI)<\exp(-\Omega(n))\Erw[Z_\beta(\PHI)]$ with high probability for {\em any} $\alpha,\beta>0$~\cite{maxsat}.
In other words, the first moment over-estimates the partition function of a typical random formula by an exponential factor.
The reason for this is a ``lottery effect'': a tiny minority of formulas render an exceptionally high contribution to $\Erw[Z_\beta(\PHI)]$.
Unsurprisingly, going to the second moment only exacerbates the problem and thus
for any $\alpha,\beta>0$ we find $\Erw[Z_\beta(\PHI)^2]\geq\exp(\Omega(n))\Erw[Z_\beta(\PHI)]^2$.
In other words, the second moment method fails rather spectacularly for all possible parameter combinations.

The first and the second moment method fall victim to similar large deviations effects in many alike ``random constraint satisfaction problems''.
These problems, ubiquitous in combinatorics, information theory, computer science and physics~\cite{nature,MM,Rudi},
can be described along the following lines.
A random {\em factor graph}, chosen either from a uniform distribution (like the random $k$-SAT model above) or from a suitable configuration model,
 induces interactions between the variables and the constraints.
The variables range over a fixed finite domain $\Omega$ and each constraint binds a few variables.
The constraints come with ``weight functions'' that either encourage or discourage certain value combinations of the incident variables.
Multiplying up the weight functions of all the contraints just like in (\ref{eqkSAT1})--(\ref{eqkSAT2}),
we obtain the Gibbs measure and the partition function.

With the standard first and second moment method drawing a blank, we seem to be at a loss as far as calculating the partition function is concerned.
However, physicists have put forward an ingenious albeit non-rigorous alternative called the {\em cavity method}~\cite{MM}.
This technique, which applies almost mechanically to any problem that can be described in the language of sparse random factor graphs,
yields an explicit conjecture as to the value of the partition function.
More specifically, the cavity method comes in several installments.
In this paper, we are concerned with the simplest, so-called ``replica symmetric'' version.

In one of their key papers~\cite{pnas} physicists hypothesized abstract conditions under which the replica symmetric cavity method
yields the correct value of the partition function.
The thrust of this paper is to prove corresponding rigorous results.
Specifically, according to \cite{pnas} the replica symmetric cavity method gives the correct answer if the Gibbs measure satisfies certain correlation decay properties.
For example, the {\em Gibbs uniqueness} condition requires that under the Gibbs measure the value assigned to a variable $x$ is asymptotically
independent of the values assigned to the variables at a large distance from $x$ in the factor graph.
In \Cor~\ref{Thm_smm} below we prove that this condition is indeed sufficient to guarantee the success of the cavity method.
Additionally, \Thm s~\ref{Thm_symUpperBound} and~\ref{Thm_nonReconstruction} yield rigorous sufficient conditions in terms of substantially weaker conditions, namely
	a symmetry property and the non-reconstruction property.

A key feature of the paper is that we establish these results not for specific examples but generically for a very wide class of factor graph models.
Of course, stating and proving general results requires a degree of abstraction.
In particular, we resort to the framework of local weak convergence of graph sequences~\cite[Part~4]{Lovasz}.
This framework suits the physics predictions well, which come in terms of the ``limiting tree'' that describes the local structure of a large random factor graph.
To be precise, 
the replica symmetric prediction is given by a functional called the {\em Bethe free energy} applied to an (infinite) random tree.

The principal tool to prove these results is a theorem about the structure of probability measures on sets of
the form $\Omega^n$ for some fixed finite set $\Omega$ and a large integer $n$, \Thm~\ref{Thm_decomp} below.
We expect that this result, which is inspired by \Szemeredi's regularity lemma~\cite{Szemeredi},  will be of independent interest.
To prove our results about random factor graphs, we combine \Thm~\ref{Thm_decomp} with the theory of local weak convergence to carry out completely
generically ``smart'' first and second moment arguments that avoid the lottery effects that the standard arguments fall victim to.

In \Sec~\ref{Sec_hom} we begin with the abstract results about probability measures on cubes.
Subsequently, in \Sec~\ref{Sec_factorGraphs} we set the stage by introducing the formalism of factor graphs and local weak convergence.
Further, in \Sec~\ref{Sec_Bethe} we state and prove the main results about Gibbs measures on random factor graphs.
Finally, \Sec~\ref{Sec_crazyConfigs} contains the proof of a technical result that enables us to control the local structure of random factor graphs.

\subsection*{Related work}
A detailed (non-rigorous) discussion of the  cavity method can be found in~\cite{MM}.
It is known that the replica symmetric version of the cavity method does not always yield the correct value of the partition function.
For instance, in some factor graph models there occurs a ``condensation phase transition'' beyond which the replica symmetric prediction is off~\cite{Lenka,pnas}.
The more complex ``1-step replica symmetry breaking (1RSB)'' version of the cavity method~\cite{MPZ} is expected to
yield the correct value of the partition function some way beyond condensation.
However, another phase transition called ``full replica symmetry breaking'' spells doom on even the 1RSB cavity method~\cite{MM}.

The replica symmetric cavity method has been vindicated rigorously in various special cases.
For instance, Montanari and Shah~\cite{MontanariShah} proved that in the random $k$-SAT model the
replica symmetric prediction is correct up to the Gibbs uniqueness threshold.
A similar result was obtained by
Bandyopadhyay and Gamarnik~\cite{Bandyopadhyay} for graph colorings and independent sets.
Furthermore, Dembo, Montanari and Sun~\cite{DMS} proved the replica symmetric conjecture on a class of models with specific types of constraints.
A strength of~\cite{DMS} is that the result applies even to sequences of non-random factor graphs under a local weak convergence assumption.
But both~\cite{DMS,MontanariShah} are based on the ``interpolation method''~\cite{FL,Guerra,PT}, which entails substantial restrictions on the types of models that can be handled.
By contrast, the present proof method is based on a completely different approach
centered around the abstract classification of measures on cubes that we present in \Sec~\ref{Sec_hom}.

Since the ``vanilla'' second moment method fails on the random $k$-SAT model, more sophisticated variants have been proposed.
The basic idea is to apply the second moment method not to the partition function itself but to a tweaked random variable.
For instance,  Achlioptas and Moore~\cite{nae} applied the second moment method to NAE-satisfying assignments, i.e.,
both the assignment and its binary inverse satisfy all clauses.
However, the number of NAE-satisfying assignments is exponentially smaller than the total number of satisfying assignments
and thus this type of argument cannot yield the typical value of the partition function.
The same is true of the more subtle random variable of Achlioptas and Peres~\cite{yuval}.
Furthermore, the work of Ding, Sly and Sun~\cite{DSS3} 
that yields the precise $k$-SAT threshold for large $k$ is based on applying the second moment method
to a random variable whose construction is guided by the 1RSB cavity method.
Among other things, the random variable from~\cite{DSS3} incorporates conditioning on the local structure of the factor graph,
an idea that will be fundamental to our arguments as well.

\subsection*{Notation}
If $\cX$ is a finite set, then we denote by $\cP(\cX)$ the set of probability measures on $\cX$.
Moreover, $\TV\nix$ signifies the total variation norm.
If $\mu$ is a probability measure on a product space $\cX^V$ for finite sets $\cX$, $V$ and $S\subset V$, then $\mu_{\marg S}\in\cP(\cX^S)$ denotes the marginal
distribution of $\mu$ on $S$.
That is, if $(x_s)_{s\in S}\in\cX^S$, then
	$$\mu_{\marg S}((x_s)_{s\in S})=\sum_{(x_s)_{s\in V\setminus S}\in\cX^{V\setminus S}}\mu((x_s)_{s\in V}).$$
If $S=\{v\}$ for some $v\in V$, then we briefly write $\mu_{\marg v}$ rather than $\mu_{\marg\{v\}}$.

The entropy of a probability measure $\mu\in\cP(\cX)$ is denoted by $H(\mu)$.
Thus, with the convention that $0\ln0=0$ we have $H(\mu)=-\sum_{x\in\cX}\mu(x)\ln\mu(x)$.
Further, agreeing that $0\ln\frac00=0$ as well, we recall that the Kullback-Leibler divergence of $\mu,\nu\in\cP(\cX)$ is
	\begin{align*}
	\KL\nu\mu&=\sum_{x\in\cX}\nu(x)\ln\frac{\nu(x)}{\mu(x)}\in[0,\infty].
	\end{align*}

We are going to work with probability measures on sets $\Omega^n$ for a (small) finite $\Omega$ and a large integer $n$ a lot.
If $\mu\in\cP(\Omega^n)$, then we write $\SIGMA_\mu,\TAU_\mu$ for two independent samples from $\mu$.
Where $\mu$ is obvious from the context we just write $\SIGMA,\TAU$.
Additionally, if $X(\SIGMA)$ is a random variable, then $\bck{X(\SIGMA)}_\mu=\sum_{\sigma\in\Omega^n}\mu(\sigma)X(\sigma)$ stands
for the expectation of $X$ with respect to $\mu$.
Further, if $\sigma\in\Omega^n$, $\emptyset\neq S\subset[n]$ and $\omega\in\Omega$, then we let
	$$\sigma[\omega|S]={|\sigma^{-1}(\omega)\cap S|}/{|S|}.$$
Thus, $\sigma[\nix|S]$ is a probability distribution on $\Omega$, namely the distribution of $\sigma(\vec x)$ for a random $\vec x\in S$.
If $S=\{x\}$ for some $x\in[n]$, then we just write $\sigma[\omega|x]$ rather than $\sigma[\omega|\{x\}]$.
Clearly, $\sigma[\omega|x]=\vecone\{\sigma(x)=\omega\}$.

We use the $\bck\nix_\mu$ notation for averages over $\mu\in\cP(\Omega^n)$ to avoid confusion with averages over other, additional random quantities,
for which we reserve the common symbols $\Erw[\nix]$, $\pr[\nix]$.
Furthermore, we frequently work with conditional expectations.
Hence, let us recall that for a probability space $(\cX,\cA,\pr)$, a random variable $X:\cX\to\RR$ and a $\sigma$-algebra $\cF\subset\cA$ the conditional expectation
$\Erw[X|\cF]$ is a $\cF$-measurable random variable on $\cX\to\RR$ such that for every $\cF$-measurable event $F$ we have
	$\Erw[\vecone\{F\}\Erw[X|\cF]]=\Erw[\vecone\{F\}X]$.
Moreover, recall that the conditional variance is defined as $\Var[X|\cF]=\Erw[X^2|\cF]-\Erw[X|\cF]^2$.

In line with the two previous paragraphs, if $Y:\Omega^n\to\RR$ is a random variable, $\mu\in\cP(\Omega^n)$ and $\cF$ is a $\sigma$-algebra on $\Omega^n$,
then we write $\bck{Y|\cF}_\mu$ for the conditional expectation, which is a $\cF$-measurable random variable
	$\sigma\in\Omega^n\mapsto\bck{Y|\cF}_\mu(\sigma)$.
Accordingly, for an event $A\subset\Omega^n$ with $\mu(A)>0$ we write $\bck{Y|A}_\mu=\bck{Y\vecone\{A\}}_\mu/\mu(A)\in\RR$
for the expectation of $Y$ given
$A$.

\section{Probability measures on the cube}\label{Sec_hom}

\noindent
In this section we present a general ``regularity lemma'' for probability measures on sets $\Omega^n$ for some finite set $\Omega$ and a large integer $n$
	(\Thm~\ref{Thm_decomp} below).

\subsection{Examples}
Needless to say, probability distributions on sets $\Omega^n$ for a small finite $\Omega$ and a large integer $n$ are ubiquitous.
To get an idea of what we might hope to prove about them in general, let us look at a few examples.

The simplest case certainly is a product measure $\mu=p^{\tensor n}$ with $p\in\cP(\Omega)$.
By the Chernoff bound, for any fixed $\eps>0$ there is $n_0=n_0(\eps,\Omega)>0$ such that for $n>n_0$ we have
	\begin{align}\label{eqApprox1}
	\bck{\TV{\SIGMA[\nix|S]-p}}_\mu&<\eps&\mbox{for every $S\subset[n]$ such that $|S|\geq\eps n$}.
	\end{align}
In words, if we fix a large enough set $S$ of coordinates and then choose $\SIGMA$ randomly, then
with probability close to one the empirical distribution on $S$ will be close to $p$.

As a twist on the previous example, let $p\in\cP(\Omega)$, assume that $n$ is a square and define a measure $\mu$
by letting
	\begin{align*}
	\mu(\omega_1,\ldots,\omega_n)&=\prod_{i=0}^{\sqrt n-1}\brk{p(\omega_{1+i\sqrt n})
			\vecone\{\forall j\in[\sqrt n]:\omega_{j+i\sqrt n}=\omega_{1+i\sqrt n}\}}.
	\end{align*} 
In words, the coordinates come in blocks of size $\sqrt n$.
While the values of all the coordinates in one block coincide and have distribution $p$, the coordinates in different blocks are independent.
Although $\mu$ is not a product distribution,  (\ref{eqApprox1}) is satisfied for any fixed $\eps>0$ and large enough $n$.
Furthermore, if for a fixed $k>1$ we choose $\vec x_1,\ldots,\vec x_k\in[n]$ uniformly and independently, then 
	\begin{align}\label{eqApprox2}
	\Erw\TV{\mu_{\marg\{\vec x_1,\ldots,\vec x_k\}}-
		\mu_{\marg\vec x_1}\tensor\cdots\tensor\mu_{\marg\vec x_k}}&<\eps,
	\end{align}
provided that $n>n_1(\eps,k,\Omega)$ is sufficiently large.
This is because for large enough $n$ it is unlikely that two of the randomly chosen $\vec x_1,\ldots,\vec x_k$ belong to the same block.

As a third example, consider the set $\Omega=\{0,1\}$ and the measure $\mu$ defined by
	\begin{align*}
	\mu^{(0)}(\omega_1,\ldots,\omega_n)&=\bcfr13^{\sum_{i=1}^n\omega_i}\bcfr23^{n-\sum_{i=1}^n\omega_i},&
	\mu^{(1)}(\omega_1,\ldots,\omega_n)&=\bcfr23^{\sum_{i=1}^n\omega_i}\bcfr12^{n-\sum_{i=1}^n\omega_i},&
	\mu&=\frac12(\mu^{(0)}+\mu^{(1)}).
	\end{align*}
All the marginals $\mu_{\marg i}$, $i\in[n]$, are equal to the uniform distribution on $\{0,1\}$.
But of course the uniform distribution on $\Omega^n$ is a horrible approximation to $\mu$.
Indeed, by the Chernoff bound with overwhelming probability a point $(\omega_1,\ldots,\omega_n)$ drawn from $\mu$
either satisfies $\frac1n\sum_{i=1}^n\omega_i\sim1/3$ or $\frac1n\sum_{i=1}^n\omega_i\sim2/3$.
However, the {\em conditional} distribution given, say, $\frac1n\sum_{i=1}^n\omega_i\leq1/2$,  is close to a product measure.
Thus, $\mu$ induces a decomposition of $\Omega^n$ into two ``states''
$S_0=\{\frac1n\sum_{i=1}^n\omega_i\leq1/2\}$, $S_1=\{\frac1n\sum_{i=1}^n\omega_i>1/2\}$ such that
 $\mu[\nix|S_0]$, $\mu[\nix|S_1]$ are close to product measures.

As a final example, consider $\Omega=\{0,1\}$, assume that $n$ is even and define $\mu\in\cP(\Omega^n)$ by letting
	\begin{align*}
	\mu(\omega_1,\ldots,\omega_n)&=\bcfr12^{n/2}\bcfr{1}{3}^{\sum_{i>n/2}\omega_i}\bcfr{2}{3}^{n/2-\sum_{i>n/2}\omega_i}.
	\end{align*}
In words, $\mu$ is a product measure with marginal distribution $\Be(1/2)$ on the first $n/2$ coordinates and $\Be(1/3)$ on the other coordinates.
Clearly, $\mu$ satisfies (\ref{eqApprox1}) with $p=\Be(1/2)$ for sets $S\subset[n/2]$ and with $p=\Be(1/3)$ for sets $S\subset[n]\setminus [n/2]$,
provided that $n$ is large.

In summary, the following picture emerges.
The conditions (\ref{eqApprox1}) and (\ref{eqApprox2}) are proxies for saying that a given measure $\mu$ resembles a product measure.
Furthermore, in order to obtain from a given $\mu$ measures that satisfy (\ref{eqApprox1}) or (\ref{eqApprox2})
it may be necessary to decompose the space $\Omega^n$ into ``states'' so that the conditional distributions have these properties.
In addition, because different coordinates may have different marginal distributions, for (\ref{eqApprox1}) to hold it may be necessary to partition
the set $[n]$ of coordinates.

\subsection{Homogeneity}
The main result of this section shows that by partitioning the space $\Omega^n$ and/or the set $[n]$ of coordinates it is always
possible to ``approximate'' a given measure $\mu$ by measures that satisfy (\ref{eqApprox1}) 
for some suitable $p$ as well as (\ref{eqApprox2}).
In fact, the number of parts that we have to partition $[n]$ and $\Omega^n$ into is bounded only in terms of the desired accuracy
but independently of $n$.

Let us introduce some terminology.
If $\vec V=(V_1,\ldots,V_k)$ is a partition of some set $V$, then we call $\#\vec V=k$ the \bemph{size} of $\vec V$.
Furthermore, a partition $\vec W=(W_1,\ldots,W_l)$ \bemph{refines} another partition $\vec V=(V_1,\ldots,V_k)$
if for each $i\in[l]$ there is $j\in[k]$ such that $W_i\subset V_j$.

For $\eps>0$ we say that $\mu\in\cP(\Omega^n)$ is \bemph{$\eps$-regular} on a set $U\subset[n]$ if
for every subset $S\subset U$ of size $|S|\geq\eps|U|$ we have
$$\bck{\TV{\SIGMA[\nix|S]-\SIGMA[\nix|U]}}_{\mu}<\eps.$$
Further, $\mu$ is \bemph{$\eps$-regular}  with respect to a partition $\vec V$ if
	there is a set $J\subset[\#\vV]$ such that $\sum_{i\in[\#\vV]\setminus J}|V_i|<\eps n$ and such that $\mu$ is $\eps$-regular on $V_i$ for all $i\in J$.
Additionally, if $\vec V$ is a partition of $[n]$ and $\vec S$ is a partition of $\Omega^n$, then
we say that $\mu$ is \bemph{$\eps$-homogeneous} with respect to $(\vec V,\vec S)$ if there is a subset $I\subset[\#\vec S]$ such that the following is true.
\begin{description}
\item[HM1] We have $\mu(S_i)>0$ for all $i\in I$ and $\sum_{i\in[\#\vS]\setminus I}\mu(S_i)<\eps$.
\item[HM2] for all $i\in[\#\vec S]$ and $j\in[\#\vec V]$ we have
	$\max_{\sigma,\sigma'\in S_i}\TV{\sigma[\nix|V_j]-\sigma'[\nix|V_j]}<\eps.$
\item[HM3]  for all $i\in I$ the measure $\mu[\nix|S_i]$ is $\eps$-regular with respect to $\vec V$.
\item[HM4] $\mu$ is $\eps$-regular with respect to $\vV$.
\end{description}

\begin{theorem}\label{Thm_decomp}
For any $\eps>0$ there exists $N=N(\eps,\Omega)>0$ such that for every $n>N$, any measure $\mu\in\cP(\Omega^n)$
and any partition $\vV_0$ of $[n]$ of size $\#\vV_0\leq1/\eps$ the following is true.
There exist a refinement $\vV$ of $\vV_0$ and a partition $\vS$ of $\Omega^n$ such that
$\#\vec V+\#\vec S\leq N$ and such that
$\mu$ is $\eps$-homogeneous with respect to $(\vec V,\vec S)$.
\end{theorem}

Informally speaking, \Thm~\ref{Thm_decomp} shows that any probability measure $\mu\in\cP(\Omega^n)$  admits a partition $(\vec V,\vec S)$
such that the following is true.
Almost the entire probability mass of $\mu$ belongs to parts $S_i$ such that the conditional measure $\mu[\nix|S_i]$ is $\eps$-regular w.r.t.\ $\vec V$.
This means that almost every coordinate $x\in[n]$ belongs to a class $V_j$ such that
for every ``large'' $U\subset V_j$ for
$\SIGMA$ chosen from $\mu[\nix|S_i]$ very likely the empirical distribution $\SIGMA[\nix|U]$ is close to the 
marginal distribution $\bck{\SIGMA[\nix|V_j]}_{\mu[\nix|S_i]}$ of the entire class.

\Thm~\ref{Thm_decomp} and its proof, which we defer to \Sec~\ref{Sec_decomp}, are inspired by \Szemeredi's regularity lemma~\cite{Szemeredi}.
Let us proceed to state a few consequences of \Thm~\ref{Thm_decomp}.

A \bemph{$(\eps,k)$-state} of $\mu$ is a set $S\subset\Omega^n$ such that $\mu(S)>0$ and
	\begin{align*}
	\frac1{n^k}\sum_{x_1,\ldots,x_k\in[n]}\TV{\mu_{\downarrow\{x_1,\ldots,x_k\}}[\nix|S]-\mu_{\downarrow x_1}[\nix|S]\tensor
		\cdots\tensor \mu_{\downarrow x_k}[\nix|S]}<\eps.
	\end{align*}
In other words, if we choose $\vec x_1,\ldots,\vec x_k\in[n]$ independently and uniformly at random, then the expected total variation distance
between the joint distribution $\mu_{\downarrow\{\vec x_1,\ldots,\vec x_k\}}[\nix|S]$ of $\vec x_1,\ldots,\vec x_k$
and the product
$\mu_{\downarrow\vec x_1}[\nix|S]\tensor\cdots\tensor \mu_{\downarrow \vec x_k}[\nix|S]$
 of the marginal distributions is small.

\begin{corollary}\label{Thm_states}
For any $\eps>0$, $k\geq2$ there exists $\eta=\eta(\eps,k,\Omega)>0$ such that for every $n>1/\eta$ any measure $\mu\in\cP(\Omega^n)$
has pairwise disjoint $(\eps,k)$-states $S_1,\ldots,S_N$ such that
	$\mu(S_i)\geq\eta$ for all $i\in[N]$ and $\sum_{i=1}^N\mu(S_i)\geq1-\eps$.
\end{corollary}

\noindent
Thus, we can chop the space $\Omega^n$ into subsets $S_1,\ldots,S_N$, $N\leq1/\eta$, that capture almost the entire probability mass
such that $\mu[\nix|S_i]$ ``resembles a product measure'' for each $i\in[N]$.
We prove \Cor~\ref{Thm_states} in \Sec~\ref{Sec_states}.

Let us call $\mu$ \bemph{$(\eps,k)$-symmetric} if $S=\Omega^n$ itself is an $(\eps,k)$-state.

\begin{corollary}\label{Cor_states}
For any $\eps,k$ there exist $\delta,\eta>0$ such that for all $n>1/\eta$ and all $\mu\in\cP(\Omega^n)$ the following is true.
	If for any two $(\delta,k)$-states $S_1,S_2$ with $\mu(S_1),\mu(S_2)\geq\eta$ we have
	\begin{equation}\label{eqCor_statesAssumption}
	\frac1n\sum_{x\in[n]}\TV{\mu_{\downarrow x}[\nix|S_1]-\mu_{\downarrow x}[\nix|S_2]}<\delta,
	\end{equation}
	then $\mu$ is $(\eps,k)$-symmetric.
\end{corollary}

\noindent
Thus, the entire measure $\mu$ ``resembles a product measure'' if extensive states have similar marginal distributions.
Conversely, we have the following.

\begin{corollary}\label{Cor_states2}
For any $\eps>0$ there is $\gamma >0$ such that for any $\eta>0$ there exists $\delta>0$ such that for all $n>1/\delta$ and all $\mu\in\cP(\Omega^n)$ the following is true.
	If $\mu$ is $(\delta,2)$-symmetric, then for any $(\gamma,2)$-state $S$ with $\mu(S)\geq\eta$ we have
	$$\frac1n\sum_{x\in[n]}\TV{\mu_{\downarrow x}[\nix|S]-\mu_{\downarrow x}}<\eps.$$
\end{corollary}

The proofs of Corollaries~\ref{Cor_states} and~\ref{Cor_states2} can be found in \Sec s~\ref{Sec_Cor_states} and~\ref{Sec_Cor_states2}, respectively.
Finally, in \Sec~\ref{Sec_Prop_tensorise} we prove the following fact that will be useful in \Sec~\ref{Sec_Bethe}.

\begin{proposition}\label{Prop_tensorise}
For any $\eps>0$ there exist $\delta>0$ such that for large enough $n$ the following is true.
If $\mu\in\cP(\Omega^n)$ is $(\delta,2)$-symmetric, then $\mu\tensor\mu\in\cP(\Omega^n\times\Omega^n)$ is $(\eps,2)$-symmetric.
\end{proposition}

\subsection{Proof of \Thm~\ref{Thm_decomp}}\label{Sec_decomp}
Throughout this section we assume that $n$ is sufficiently large.
To prove \Thm~\ref{Thm_decomp} and guided by~\cite{Szemeredi}, we define the \bemph{index} of $\mu$ with respect to a partition $\vec V$ of $[n]$ as
	$$\ind_\mu(\vec V)=\frac{1}{|\Omega|n}\sum_{\omega\in\Omega}\sum_{j\in[\#\vec V]}\sum_{x\in V_j}
		\bck{(\SIGMA[\omega|x]-\SIGMA[\omega|V_j])^2}_\mu.$$
The index can be viewed as a conditional variance (cf.\ \cite{Tao}).
Indeed, choose $\vec x\in[n]$ uniformly and independently of $\SIGMA$.
Furthermore, let $\cF_{\vV}$ be the $\sigma$-algebra generated by 
the events $\{\vec x\in V_i\}$ for $i\in[\#\vV]$.
Writing $\Erw[\nix]$ and $\Var[\nix]$ for the expectation and variance with respect to the choice of $\vec x$ only, we see that
	$$\ind_\mu(\vV)=\frac1{|\Omega|}\sum_{\omega\in\Omega}\Erw\bck{\Var[\SIGMA[\omega|\vec x]|\cF_{\vV}]}_\mu.$$

\begin{lemma}\label{Lemma_refinement}
For any partition $\vec V$ of $[n]$ we have $\ind_\mu(\vec V)\in[0,1]$.
If $\vec W$ is a refinement of $\vec V$, then $\ind_\mu(\vec W)\leq\ind_\mu(\vec V)$.
\end{lemma}
\begin{proof}
The fact that $\ind_\mu(\vec V)\in[0,1]$ is immediate from the definition.
Moreover, if $\vW$ refines $\vV$, then $\cF_{\vV}\subset\cF_{\vW}$.
Consequently,
	$\Erw\bck{\Var[\SIGMA[\omega|\vec x]|\cF_{\vW}]}_\mu\leq\Erw\bck{\Var[\SIGMA[\omega|\vec x]|\cF_{\vV}]}_\mu$.
Averaging over $\omega\in\Omega$ yields $\ind_\mu(\vec W)\leq\ind_\mu(\vec V)$.
\end{proof}

\begin{lemma}\label{Lemma_homreg}
If $\mu\in\cP(\Omega^n)$ fails to be $\eps$-regular with respect to $\vec V$, then
there is a refinement $\vec W$ of $\vec V$ such that $\#\vec W\leq2\#\vec V$ and
	$\ind_\mu(\vec W)\leq\ind_\mu(\vec V)-\eps^4/(4|\Omega|^3).$
\end{lemma}
\begin{proof}
Let $\bar J$ be the set of all indices $j\in[\#\vec V]$ such that
there exists $S\subset V_j$ of size $|S|\geq\eps|V_j|$ such that 
	\begin{equation}\label{eqClaim_indinc1}
	\bck{\TV{\SIGMA[\nix|S]-\SIGMA[\nix|V_j]}}_\mu\geq\eps.
	\end{equation}
Since $\mu$ fails to be $\eps$-regular with respect to $\vV$ we have
	\begin{equation}\label{eqClaim_indinc2}
	\sum_{j\in \bar J}|V_j|\geq\eps n.
	\end{equation}
For each $j\in\bar J$ pick a set $S_j\subset V_j$, $|S_j|\geq\eps|V_j|$ such that (\ref{eqClaim_indinc1}) is satisfied.
Then there exists $\omega_j\in\Omega$ such that
	\begin{equation}\label{eqClaim_indinc99}
	\bck{\abs{\SIGMA[\omega_j|S_j]-\SIGMA[\omega_j|V_j]}}_\mu\geq\eps/(2\abs\Omega).
	\end{equation}
Let $\vW$ be the partition obtained from $\vV$ by splitting each class $V_j$, $j\in\bar J$, into the sub-classes $S_j,V_j\setminus S_j$.
Clearly, $\#\vW\leq2\#\vV$.
Furthermore, 
	\begin{align}\nonumber
	\ind_\mu(\vV)&=\frac1{|\Omega|}\sum_{\omega\in\Omega}\Erw\bck{\Var[\SIGMA[\omega|\vec x]|\cF_{\vV}]}_\mu=
		\frac1{|\Omega|}\sum_{\omega\in\Omega}\bc{\Erw\bck{\Var[\SIGMA[\omega|\vec x]|\cF_{\vW}]}_\mu
			+\Erw\bck{\Var[\Erw[\SIGMA[\omega|\vec x]|\cF_{\vW}]|\cF_{\vV}]}_\mu}\\
			&=\ind_{\mu}(\vW)+\frac1{|\Omega|}\sum_{\omega\in\Omega}\Erw\bck{\Var[\Erw[\SIGMA[\omega|\vec x]|\cF_{\vW}]|\cF_{\vV}]}_\mu.
				\label{eqClaim_indinc666}
	\end{align}
If $j\in\bar J$ then (\ref{eqClaim_indinc99}) implies that on $V_j$ we have
	\begin{align}				\label{eqClaim_indinc667}
	\bck{\Var[\Erw[\SIGMA[\omega_j|\vec x]|\cF_{\vW}]|\cF_{\vV}]}_\mu&\geq
		\frac{|S_j|}{|V_j|}\bck{(\SIGMA[\omega_j|S_j]-\SIGMA[\omega_j|V_j])^2}_\mu\geq\frac{\eps^3}{4|\Omega|^2}.
	\end{align}
Hence, combining (\ref{eqClaim_indinc2}) and (\ref{eqClaim_indinc667}), we find
	\begin{align}				\label{eqClaim_indinc668}
	\frac1{|\Omega|}\sum_{\omega\in\Omega}\Erw\bck{\Var[\Erw[\SIGMA[\omega|\vec x]|\cF_{\vW}]|\cF_{\vV}]}_\mu&\geq\frac{\eps^4}{4|\Omega|^3}.
	\end{align}
Finally, the assertion follows from (\ref{eqClaim_indinc666}) and (\ref{eqClaim_indinc668}).	
\end{proof}

\begin{proof}[Proof of \Thm~\ref{Thm_decomp}]
The set $\cP(\Omega)$ is compact.
Therefore, there exists a partition $\vec Q=(Q_1,\ldots,Q_K)$ of $\cP(\Omega)$ into pairwise disjoint sets
such that for all $i\in[K]$ and any two measures $\mu,\mu'\in Q_i$ we have $\TV{\mu-\mu'}<\eps$.

Given any partition $\vec W$ of $[n]$, we can construct a corresponding decomposition $\vec S(\vec W)$ of $\Omega^n$ as follows.
Call  $\sigma,\sigma'\in\Omega^n$ $\vec W$-equivalent if for every $i\in[\#\vec W]$ there exists $j\in[\#\vec Q]$ such that
	$\sigma[\nix|W_i],\sigma'[\nix|W_i]\in Q_j$.
Then $\vec S(\vec W)$ comprises of the equivalence classes.

We construct the desired partition $\vec V$ of $[n]$ inductively, starting from any given partition $\vec V(0)$ of size at most $1/\eps$.
The construction stops once $\mu$ is $\eps$-homogeneous with respect to $(\vec V(t),\vec S(\vec V(t)))$.
Assuming that this is not the case,  we obtain $\vec V({t+1})$ from $\vec V(t)$ as follows.
If $\mu$ fails to be $\eps$-regular with respect to $\vV(t)$, then we let $\vV({t+1})$ be the partition promised by \Lem~\ref{Lemma_homreg},
which guarantees that 
	\begin{align}\label{eqSimpleStep}
	\#\vV(t+1)\leq2\#\vV(t)
		\quad\mbox{and}\quad
		\ind_\mu(\vV(t+1))\leq\ind_\mu(\vec V(t))-\eps^4/(4|\Omega|^3).
		\end{align}
Otherwise let $\vec S(t)=\vec S(\vec V(t))$ and $s(t)=\#\vec S(t)$ for the sake of brevity.
Further, let $\mu_{i,t}=\mu[\nix|S_i(t)]$ for $i\in[s(t)]$ with $\mu[S_i(t)]>0$.
Moreover, let $\bar I(t)$ be the set of all $i\in[s(t)]$ such that $\mu[S_i(t)]>0$ and $\mu_{i,t}$ fails to be $\eps$-regular with respect to $\vec V(t)$.
If $\mu$ fails to be $\eps$-homogeneous with respect to $(\vec V(t),\vec S(t))$ but $\mu$ is $\eps$-regular w.r.t.\ $\vV(t)$, then
	\begin{equation}\label{eqThm_decomp1}
	\sum_{i\in\bar I(t)}\mu[S_i(t)]\geq\eps.
	\end{equation}
\Lem~\ref{Lemma_homreg} shows that for any $i\in\bar I(t)$ there exists a refinement $\vec W(t,i)$ of $\vec V(t)$ such that 
	\begin{equation}\label{eqThm_decomp2}
	\ind_{\mu_{i,t}}(\vec W(t,i))\leq\ind_{\mu_{i,t}}(\vec V(t))-\eps^4/(4|\Omega|^3).
	\end{equation}
Let $\vec V(t+1)$ be the coarsest common refinement of all the partitions $(\vec W(t,i))_{i\in\bar I(t)}$.
Then 
	\begin{equation}\label{eqThm_decomp3a}
	\#\vec V(t+1)\leq\#\vec V(t)\cdot2^{\#\vec Q^{\#\vec V(t)}}.
	\end{equation}
In addition, (\ref{eqThm_decomp2}) and \Lem~\ref{Lemma_refinement} imply
	\begin{equation}\label{eqThm_decomp3}
	\ind_{\mu_{i,t}}(\vec V(t+1))\leq\ind_{\mu_{i,t}}(\vec V(t))-
		\vecone\{i\in\bar I(t)\}\eps^4/(4|\Omega|^3).
	\end{equation}
Therefore, by (\ref{eqThm_decomp1}), (\ref{eqThm_decomp3}) and Bayes' rule
	\begin{align}
	\ind_\mu(\vec V(t+1))&=\frac1{n|\Omega|}\sum_{\omega\in\Omega}\sum_{j\in[\#\vec V(t+1)]}\sum_{x\in V_j(t+1)}
		\bck{(\SIGMA[\omega|x]-\SIGMA[\omega|V_j(t+1)])^2}_\mu\nonumber\\
		&=\frac1{n|\Omega|}\sum_{\omega,j,x}	\sum_{i\in[s(t)]:\mu[S_i(t)]>0}\mu[S_i(t)]
			\bck{(\SIGMA[\omega|x]-\SIGMA[\omega|V_j(t+1)])^2}_{\mu_{i,t}}\nonumber\\
		&=\sum_{i:\mu[S_i(t)]>0}\mu[S_i(t)]\ind_{\mu_{i,t}}(\vec V(t+1))\nonumber\\
		&\leq-\eps^5/(4|\Omega|^3)+\sum_{i:\mu[S_i(t)]>0}\mu[S_i(t)]\ind_{\mu_{i,t}}(\vec V(t))=\ind_\mu(\vec V(t))-\eps^5/(4|\Omega|^3).
			\label{eqThm_decomp4}
	\end{align}

Combining (\ref{eqSimpleStep}), (\ref{eqThm_decomp4}) and \Lem~\ref{Lemma_refinement}, we conclude that $\mu$ is $\eps$-homogeneous
with respect to $(\vec V(T),\vec S(T))$ for some $T\leq4|\Omega|^3/\eps^5$.
Finally, (\ref{eqThm_decomp3a}) entails that $\#\vec V(T),\#\vec  S(T)$ are bounded in terms of $\eps,\Omega$ only.
\end{proof}

\subsection{Proof of \Cor~\ref{Thm_states}}\label{Sec_states}
To derive \Cor~\ref{Thm_states} from \Thm~\ref{Thm_decomp} we use the following handy sufficient condition for $(\eps,k)$-symmetry.

\begin{lemma}\label{Lemma_regularSymmetric}
For any $k\geq2$, $\eps>0$ there is $\delta=\delta(\eps,k,\Omega)$ such that for large enough $n$ the following is true.
Assume that $\mu\in\cP(\Omega^n)$ is $\delta$-regular with respect to a partition $\vec V$ and 
set $\bar\mu_i(\nix)=\bck{\SIGMA[\nix|V_i]}_\mu$ for $i\in[\#\vV]$.
If
	\begin{align}\label{eqLemma_regularSymmetric0}
	\sum_{i\in[\#\vec V]}\frac{|V_i|}{n}\bck{\TV{\SIGMA[\nix|V_i]-\bar\mu_i}}_\mu<\delta,
	\end{align}
then $\mu$ is $(\eps,k)$-symmetric.
\end{lemma}
\begin{proof}
Choose a small $\xi=\xi(\eps,k,\Omega)>0$ and a smaller $\delta=\delta(\xi)>0$.
Then (\ref{eqLemma_regularSymmetric0}) implies that there is $J\subset[\#\vec V]$ satisfying
	\begin{equation}\label{eqLemma_regularSymmetric1}
	\sum_{j\in J}|V_j|\geq(1-\xi)n
	\end{equation}
such that for all $j\in J$, $S\subset V_j$, $|S|\geq\xi|V_j|$ we have
	\begin{equation}\label{eqLemma_regularSymmetric2}
	\bck{\TV{\SIGMA[\nix|S]-\bar\mu_j}}_\mu\leq\xi.
	\end{equation}
In particular, we claim that (\ref{eqLemma_regularSymmetric2}) implies the following (if $\xi$ is small enough):
	\begin{equation}\label{eqMarkovAnwenden}
	\forall \omega\in\Omega, j\in J,\Sigma\subset\Omega^n:\mu(\Sigma)\geq\xi^{1/4}\Rightarrow
			\abs{\cbc{x\in V_j:\abs{\bck{\SIGMA[\omega|x]|\Sigma}_\mu-\bar\mu_j(\omega)}>\xi^{1/4}}}\leq\xi^{1/4}|V_j|.
	\end{equation}
Indeed, assume that $\bck{\vecone\{\SIGMA\in\Sigma\}}_\mu\geq\xi^{1/4}$ and 
$\abs{\cbc{x\in V_j:\abs{\bck{\SIGMA[\omega_0|x]|\Sigma}_\mu-\bar\mu_j(\omega_0)}>\xi^{1/4}}}>\xi^{1/4}|V_j|$ for some $\omega_0\in\Omega$.
Then because $\bck{\SIGMA[\nix|x]|\Sigma}_\mu$ is a probability measure on $\Omega$ for every $x$, there exists $\omega\in\Omega$ such that
the set $S=\cbc{x\in V_j:\bck{\SIGMA[\omega|x]|\Sigma}_\mu<\bar\mu_j(\omega)-\xi^{1/4}/|\Omega|}$ has size $|S|>\xi^{1/4}|V_j|/(2|\Omega|)$.
In particular,
	$\bck{\SIGMA[\omega|S]|\Sigma}_\mu\leq \bar\mu_j(\omega)-\xi^{1/4}/|\Omega|.$
Therefore, by Markov's inequality
	$$\bck{\vecone\{\SIGMA[\omega|S]\geq \bar\mu_j(\omega)-\xi^{1/3}\}|\Sigma}_\mu\leq\frac{\bar\mu_j(\omega)-\xi^{1/4}/|\Omega|}{\bar\mu_j(\omega)-\xi^{1/3}}
		\leq\frac{1-\xi^{1/4}/|\Omega|}{1-\xi^{1/3}}\leq1-\xi^{1/4}/(2|\Omega|).$$
Consequently, we obtain
	 $$\bck{\TV{\SIGMA[\nix|S]-\bar\mu_j}}_\mu\geq\xi^{1/3+1/4}\bck{\vecone\{\SIGMA\in\Sigma\}}_\mu/(2|\Omega|)\geq\xi^{7/8}.$$
Since $|S|>\xi^{1/4}|V_j|/(2|\Omega|)>\xi|V_j|$, this is a contradiction to (\ref{eqLemma_regularSymmetric2}).

Now, 
fix any $\omega_1,\ldots,\omega_k\in\Omega$ and
let $\vec x_1,\ldots,\vec x_k\in[n]$ be chosen independently and uniformly at random.
Let $\Sigma_h=\Sigma_h(\vec x_1,\ldots,\vec x_h)\subset\Omega^n$ be the event that $\SIGMA(\vec x_i)=\omega_i$ for all $i\leq h$. 
We are going to show that for $0\leq h<k$,
	\begin{equation}\label{eqLemma_regularSymmetric3}
	\Erw\brk{\mu(\Sigma_h)\abs{\bck{\SIGMA[\omega_{h+1}|\vec x_{h+1}]|\Sigma_h}_\mu
		-\bck{\SIGMA[\omega_{h+1}|\vec x_{h+1}]}_\mu}}<\xi^{1/5}.
	\end{equation}
In the case $h=0$ there is nothing to show.
As for the inductive step,  condition on $\vec x_1,\ldots,\vec x_{h}$.
\begin{description}
\item[Case 1: $\mu(\Sigma_h)\leq\xi^{1/4}$]
	regardless of the choice of $\vec x_{h+1}$ we have
		\begin{align*}
		\mu(\Sigma_h)\abs{\bck{\SIGMA[\omega_{h+1}|\vec x_{h+1}]|\Sigma_h}_\mu
		-\bck{\SIGMA[\omega_{h+1}|\vec x_{h+1}}_\mu}\leq\xi^{1/4}.
		\end{align*}
\item[Case 2: $\mu(\Sigma_h)>\xi^{1/4}$]
	due to (\ref{eqLemma_regularSymmetric1}) with probability at least $1-2\xi$ we have
		$\vec x_{h+1}\in V_j\setminus\{\vec x_1,\ldots,\vec x_h\}$ for some $j\in J$.
	Hence, (\ref{eqMarkovAnwenden}) implies
		$\Erw_{\vec x_{h+1}} \left[ \abs{\bck{\SIGMA[\omega_{h+1}|\vec x_{h+1}]|\Sigma_h}_\mu-\bck{\SIGMA[\omega_{h+1}|\vec x_{h+1}]}_\mu}\right] 
		\leq\xi^{1/4}.$
\end{description}
Hence, (\ref{eqLemma_regularSymmetric3}) follows.

To complete the proof, we are going to show by induction on $h\in[k]$ that
	\begin{align}\label{eqNasty1}
	\Erw\abs{\bck{\prod_{i=1}^h\SIGMA[\omega_i|\vec x_i]}_\mu-\prod_{i=1}^h\bck{\SIGMA[\omega_i| \vec x_i]}_\mu}&\leq h \xi^{1/5}.
	\end{align}
For $h=1$ there is nothing to show.
To proceed from $h$ to $h+1$ we use the triangle inequality to write
	\begin{align*}
	\Erw\brk{
	\abs{\bck{\prod_{i=1}^{h+1}\SIGMA[\omega_i|\vec x_i]}_\mu-\prod_{i=1}^{h+1}\bck{\SIGMA[\omega_i|\vec x_i]}_\mu}}
	\leq & \Erw\brk{\mu(\Sigma_h)\abs{\bck{\SIGMA[\omega_{h+1}|\vec x_{h+1}]|\Sigma_h}_\mu
		-\bck{\SIGMA[\omega_{h+1}|\vec x_{h+1}]}_\mu}}
	\\ &+ \Erw\brk{
	\bck{\SIGMA[\omega_{h+1}|\vec x_{h+1}]}_\mu \abs{ \bck{\prod_{i=1}^{h}\SIGMA[\omega_i|\vec x_i]}_\mu-\prod_{i=1}^{h}\bck{\SIGMA[\omega_i|\vec x_i]}_\mu}}.
	\end{align*}
Invoking the induction hypothesis and (\ref{eqLemma_regularSymmetric3}) completes the proof.
\end{proof}

\begin{proof}[Proof of \Cor~\ref{Thm_states}]
For a small enough $\delta=\delta(\eps,k)>0$ let $(\vV,\vS)$
be a pair of partitions of size at most $N=N(\delta,\Omega)$ such that $\mu$ is $\delta/2$-homogeneous with respect to $(\vV,\vS)$ as guaranteed by \Thm~\ref{Thm_decomp}.
Let $\eta=\eps/(2N)$ and let $J$ be the set of all $j\in[\#\vS]$ such that $\mu(S_j)\geq\eta$ and such that $\mu[\nix|S_j]$ is $\delta$-regular with respect to $\vV$.
Then
	\begin{align*}
	\sum_{j\in[\#\vS]\setminus J}\mu(S_j)&\leq\delta+\eps/2<\eps.
	\end{align*}
Furthermore, for every $j\in J$ the measure $\mu[\nix|S_j]$ satisfies (\ref{eqLemma_regularSymmetric0}) due to {\bf HM2}.
Therefore, \Lem~\ref{Lemma_regularSymmetric} implies that $\mu[\nix|S_j]$ is $(\eps,k)$-symmetric.
Consequently, the sets $(S_j)_{j\in J}$ are pairwise disjoint $(\eps,k)$-states  with
$\mu(S_j)\geq\eta$ for all $j\in J$ and
 $\sum_{j\in J}\mu(S_j)\geq1-\eps$.
\end{proof}

\subsection{Proof of \Cor~\ref{Cor_states}}\label{Sec_Cor_states}
Pick small enough $\delta=\delta(\eps,k,\Omega),\gamma=\gamma(\delta),\eta(\gamma)>0$.
Then by \Thm~\ref{Thm_decomp} $\mu$ is $\gamma$-homogeneous with respect to $(\vV,\vS)$ for partitions that satisfy $\#\vV+\#\vS\leq N=N(\gamma)$.
Let $J\subset[\#\vS]$ contain all $j$ such that $\mu[\nix|S_j]$ is $\gamma$-regular with respect to $\vV$ and such that $\mu(S_j)\geq\eta$.
Let $\bar\mu_{i,j}=\bck{\SIGMA[\nix|V_i]}_{\mu[\nix|S_j]}$. 
Then by {\bf HM2} for every $j\in J$ we have
	\begin{align*}
	\frac1{n}\sum_{i\in[\#\vec V]}|V_i|\bck{\TV{\SIGMA[\nix|V_i]-\bar\mu_{i,j}}}_{\mu[\nix|S_j]}<3\gamma.
	\end{align*}
Therefore, \Lem~\ref{Lemma_regularSymmetric} implies that $S_j$ is a $(\delta,2)$-state.
Consequently, our assumption (\ref{eqCor_statesAssumption}) and the triangle inequality entail that for all $j,j'\in J$,
	\begin{align}\label{eqInTheMiddle}
	\sum_{i\in[\#\vV]}\frac{|V_i|}n\TV{\bck{\SIGMA[\nix|V_i]}_{\mu[\nix|S_j]}-\bck{\SIGMA[\nix|V_i]}_{\mu[\nix|S_{j'}]}}&<\delta.
	\end{align}
Choosing $\eta$ small, we can ensure that $\sum_{j\not\in J}\mu(S_j)\leq\delta$.
Therefore, letting $\bar\mu_i=\bck{\SIGMA[\nix|V_i]}_\mu$, we obtain from (\ref{eqInTheMiddle})
	\begin{align}
	\sum_{i\in[\#\vV]}\frac{|V_i|}n\bck{\TV{\SIGMA[\nix|V_i]-\bar\mu_i}}_{\mu}&
		\leq\delta+
	\sum_{i\in[\#\vV]}\frac{|V_i|}n\sum_{j\in J}\mu(S_j)\bck{\TV{\SIGMA[\nix|V_i]-\bar\mu_i}}_{\mu[\nix|S_j]}\nonumber\\
		&\leq2\delta+\sum_{i\in[\#\vV]}\frac{|V_i|}n\sum_{j\in J}\mu(S_j)\TV{\bck{\SIGMA[\nix|V_i]}_{\mu[\nix|S_j]}-\bar\mu_i}
			\qquad\mbox{[by {\bf HM2}]}\nonumber\\
		&\leq5\delta.\label{eqInTheEnd}
	\end{align}
Since $\mu$ is $\gamma$-regular and thus $5\delta$-regular w.r.t.\ $\vV$  by {\bf HM4}, (\ref{eqInTheEnd}) and
\Lem~\ref{Lemma_regularSymmetric} imply that $\mu$ is $(\eps,k)$-symmetric.

\subsection{Proof of \Cor~\ref{Cor_states2}}\label{Sec_Cor_states2}
Choose a small $\gamma = \gamma(\eps,\Omega)$ and a smaller $\delta = \delta(\gamma,\eta)$.
Assume that $S$ is a $(\gamma,2)$-state with $\mu(S) \geq \eta$ and that $\mu$ is $(\delta,2)$ symmetric. 
Assume for contradiction that
	\begin{equation}\label{eqCor_states2_proof1}
	\frac1n\sum_{x\in[n]}\TV{\mu_{\downarrow x}[\nix|S]-\mu_{\downarrow x}}>\eps.
	\end{equation}
Let
	\begin{align*}
	W&= \left \{x\in V: \TV{\mu_{\downarrow x}[\nix|S]-\mu_{\downarrow x}[\nix]]} \geq \eps/2 \right\}&&\mbox{and}\\
	W_{s}(\omega)&= \left \{x \in W_i: s\cdot\left( \mu_{\marg x}[\omega|S] - \mu_{\marg x}[\omega] \right) \geq {\eps}/(2|\Omega|) \right \}
		&&\mbox{for $\omega\in\Omega$, $s\in\{\pm1\}$.}
	\end{align*}
Then (\ref{eqCor_states2_proof1}) entails that  $|W| \geq \eps n/2$.
Therefore, there is $\omega \in \Omega$ such that $|W_{s}(\omega)| \geq \eps n/(2|\Omega|)$ for either $s=+1$ or $s=-1$.
Let $W' = W_{s}(\omega)$ for the sake of brevity.
Of course, by the definition of $W'$,
	\beq \label{eq_proof_cor_states2_v_0}
	\bc{\bck{\SIGMA[\omega|W']}_{\mu[\nix|S]}-\bck{\SIGMA[\omega|W']}_{\mu}}^2\geq\frac{\eps^2}{4 | \Omega|^2 }
	\eeq
Moreover, because $S$ is an $(\gamma,2)$-state, the measure $\mu[\nix|S]$ is $(\gamma,2)$-symmetric.
Therefore,
\begin{align}\nonumber
	\bck{ \bc{\SIGMA[\omega|W']-\bck{\TAU[\omega|W']}_{\mu[\nix|S]}}^2}_{\mu[\nix|S]} &=
	\frac{1}{|W'|^2} \sum_{x,y\in W'}\brk{\bck{\SIGMA[\omega|x]\SIGMA[\omega|y]}_{\mu[\nix|S]}-
		\bck{\TAU[\omega|x]}_{\mu[\nix|S]}\bck{\TAU[\omega|y]}_{\mu[\nix|S]}}\\
	&\leq \frac{4 \gamma|\Omega|^2}{\eps^2}\qquad[\mbox{as $|W'|\geq\eps n/(2|\Omega|)$]}.
	\label{eq_proof_cor_states2_v_2} 
	\end{align}
Similarly, since $\mu$ is $(\delta,2)$-symmetric,
	\begin{align}
	\bck{\bc{\SIGMA[\omega|W']-\bck{\TAU[\omega|W']}_\mu}^2}_{\mu} &= 
		\frac{1}{|W'|^2} \sum_{x,y\in W'}
			\brk{\bck{\SIGMA[\omega|x]\SIGMA[\omega|y]}_{\mu}-
		\bck{\TAU[\omega|x]}_{\mu}\bck{\TAU[\omega|y]}_{\mu}}
		 \leq \frac{4  \delta|\Omega|^2}{\eps^2}.	
		\label{eq_proof_cor_states2_v_1}
	\end{align}
On the other hand we have
	\begin{align} \nonumber 
	\bck{\bc{\SIGMA[\omega|W']-\bck{\TAU[\omega|W']}_\mu}^2}_{\mu} 
		&\geq \mu(S)  \bck{\left. \bc{\SIGMA[\omega|W']-\bck{\TAU[\omega|W']}_\mu}^2 \right.}_{\mu[\nix|S]} \\ 
		&\hspace{-3cm} \geq  \mu(S) \bc{ \frac{1}{2} \bc{\bck{\TAU[\omega|W']}_{\mu[\nix|S]}-\bck{\TAU[\omega|W']}_{\mu}}^2
			-\bck{\left. \bc{\SIGMA[\omega|W']-\bck{\TAU[\omega|W']}_{\mu[\nix|S]}}^2 \right. }_{\mu[\nix|S]} }.
			\label{eq_proof_cor_states2_v_3}
	\end{align}
Finally, plugging (\ref{eq_proof_cor_states2_v_0}), (\ref{eq_proof_cor_states2_v_1}) and (\ref{eq_proof_cor_states2_v_2}) into
 (\ref{eq_proof_cor_states2_v_3}), 
 we find
 	\begin{align*}
	\frac{4\delta|\Omega|^2}{\eps^2}&\geq\eta\brk{\frac{\eps^2}{8|\Omega|^2}-\frac{4\gamma|\Omega|^2}{\eps^2}},
	\end{align*}
which is a contradiction if $\delta$ is chosen small enough.

\subsection{Proof of \Prop~\ref{Prop_tensorise}}\label{Sec_Prop_tensorise}
Choose small enough $\alpha=\alpha(\eps,\Omega)$, $\gamma=\gamma(\alpha)>0$, $\chi = \chi(\gamma)>0$ and an even smaller $\delta=\delta(\gamma,\chi)>0$ and assume that $\mu$ is
$(\delta,2)$-symmetric.
Suppose that $\mu$ is $\chi$-homogeneous with respect to a partition $(\vV,\vS)$ such that $\#\vV+\#\vS\leq N=N(\gamma)$
as promised by \Thm~\ref{Thm_decomp}.
Let $J$ be the set of all $j\in[\#\vS]$ such that $\mu(S_j)\geq\gamma^2/N$.
Moreover,  let $I$ be the set of all $i\in[\#\vV]$ such that $\mu$ is $\chi$-regular on $V_i$ and $|V_i|\geq\gamma n/N$.
By \Cor~\ref{Cor_states2} we have
	\begin{align*}
	\frac1{|V_i|}\sum_{x\in V_i}\TV{\mu_{\marg x}[\nix|S_j]-\mu_{\marg x}[\nix]}&<\gamma&\mbox{for all }i\in I,j\in J,
	\end{align*}
provided that $\delta$ is chosen small enough.
Therefore, letting $\bar\mu_i=\bck{\SIGMA[\nix|V_i]}_\mu$, for all $i\in I$ we have
	\begin{align}\label{eqProp_tensorise1}
	\bck{\TV{\SIGMA[\nix|V_i]-\bar\mu_i}}_\mu&<2\gamma.
	\end{align}

Fix some $i\in I$.
We claim that $\mu\tensor\mu$ is $\alpha$-regular on $V_i$.
Hence, let $U\subset V_i$ be a set of size $|U|\geq\alpha|V_i|$ and let
	\begin{align*}
	\cE=\cbc{\TV{\SIGMA[\nix|U]-\bar\mu_i}\leq\gamma^{1/3}}.
	\end{align*}
Then (\ref{eqProp_tensorise1}) implies that $\bck{\vecone\{\SIGMA\not\in\cE\}}_\mu<\gamma^{1/3}$,
	because $\mu$ is $\gamma$-regular on $V_i$.
Now, fix some $\sigma\in\cE$.
For $\omega\in\Omega$ let $U(\sigma,\omega)=\{x\in U:\sigma(x)=\omega\}$.
Let
	\begin{align*}
	\cE'(\sigma,\omega)=\cbc{\TV{\TAU[\nix|U(\sigma,\omega)]-\bar\mu_i}\leq\gamma^{1/3}}.
	\end{align*}
If $|U(\sigma,\omega)|\geq\gamma^{1/2}|U|$, then due to (\ref{eqProp_tensorise1}) and $\gamma$-regularity we obtain, by a similar token as previously,
$\bck{\vecone\{\TAU\notin\cE'(\sigma,\omega)\}}_\mu\leq \gamma^{1/3}$.
Consequently, the event $\cE'(\sigma)$ that $\cE'(\sigma,\omega)$ occurs 
for all $\omega$ satisfying $|U(\sigma,\omega)|\geq\gamma^{1/2}|U|$ has probability at least $1-|\Omega|\gamma^{1/3}$.
Therefore, for any $\omega,\omega'\in\Omega$ we obtain
	\begin{align*}
	&\bck{\abs{\frac1{|U|}\sum_{x\in U}\vecone\{\SIGMA(x)=\omega\}\vecone\{\TAU(x)=\omega'\}-\mu_i(\omega)\mu_i(\omega')}}_\mu\\
		&\qquad
		\leq(|\Omega|+1)\gamma^{1/3}+
			\bck{\abs{\frac1{|U|}\sum_{x\in U}\vecone\{\SIGMA(x)=\omega\}\vecone\{\TAU(x)=\omega'\}-\mu_i(\omega)\mu_i(\omega')}|
					\SIGMA\in\cE,\TAU\in\cE'(\SIGMA)}_\mu\\
		&\qquad
			\leq\gamma^{1/7}+
			\bck{\max_{\omega:|U(\SIGMA,\omega)|\geq\gamma^{1/2}|U|}|\TAU[\omega'|U(\SIGMA,\omega)]-\mu_i(\omega')|
					\,\big|\,\SIGMA\in\cE,\TAU\in\cE'(\SIGMA)}_\mu
		\leq\gamma^{1/8}.
	\end{align*}
Summing over all $\omega,\omega'$ and choosing $\gamma$ small enough, we conclude that $\mu\tensor\mu$ is $\alpha$-regular on $V_i$.

Finally, (\ref{eqProp_tensorise1}) implies that $\mu\tensor\mu$ satisfies
	\begin{align*}
	\bck{\TV{(\SIGMA \vec \tensor \TAU) [\nix|V_i]-\bar\mu_i\tensor\bar\mu_i}}_{\mu\tensor\mu}&<\alpha.
	\end{align*}
Therefore, 
picking $\alpha$ small enough, we can apply
\Lem~\ref{Lemma_regularSymmetric} to conclude that $\mu\tensor\mu$ is $(\eps,2)$-symmetric.

\section{Factor graphs}\label{Sec_factorGraphs}

\subsection{Examples}\label{Sec_FactorGraphExamples}
The aim in this section is to set up a comprehensive framework for the study of ``random factor graphs'' and their corresponding Gibbs measures.
To get started let us ponder a few concrete examples.

In the \emph{Ising model} on a graph $G=(V,E)$ the variables of the problem are just the vertices of the graph.
The values available for each variable are $\pm1$.
Thus, an assignment is simply a map $\sigma:V\to\{\pm1\}$.
Moreover,  each edge of $G$ gives rise to a constraint.
Specifically,  given a parameter $\beta>0$ we define a weight function $\psi_e$ corresponding to the edge $e=\{v,w\}$ by letting
	$\psi_e(\sigma)=\exp(\beta\sigma(v)\sigma(w)).$
Thus, edges $e=\{v,w\}$ give larger weight to assignments $\sigma$ such that $\sigma(v)=\sigma(w)$
than in the case $\sigma(v)\neq\sigma(w)$.
The corresponding partition function reads
	\begin{align*}
	Z_\beta(G)&=\sum_{\sigma:V\to\{\pm1\}}\prod_{e\in E}\psi_e(\sigma)=\sum_{\sigma:V\to\{\pm1\}}\exp\brk{\beta\sum_{\{v,w\}\in E}\sigma(v)\sigma(w)}.
	\end{align*}
Further, the  Gibbs distribution $\mu_{G,\beta}$ induced by $G$, $\beta$ is the probability measure on $\{\pm1\}^V$ defined by
	\begin{align*}
	\mu_{G,\beta}(\sigma)&=\frac1{Z_\beta(G)}\prod_{e\in E}\psi_e(\sigma)
		=\frac1{Z_\beta(G)}\exp\brk{\beta\sum_{\{v,w\}\in E}\sigma(v)\sigma(w)}.
	\end{align*}
Thus, $\mu_{G,\beta}$ weighs assignments according to the number of edges $e=\{v,w\}$ such that $\sigma(v)=\sigma(w)$.

The Ising model has been studied extensively in the mathematical physics literature on various classes of graphs, including and particularly random graphs.
For instance, if $\G(n,d)$ is a random regular graph of degree $d$ on $n$ vertices, then 
$Z_\beta(\G(n,d))$ is known to ``converge'' to the value predicted by the cavity method~\cite{DM}.
Formally, the cavity method yields a certain number $F(\beta,d)$ such that 
	\begin{equation}\label{eqIsingConv}
	\lim_{n\to\infty}\frac1n\Erw[\ln Z_\beta(\G(n,d))]=F(\beta,d).
	\end{equation}
Because $Z_\beta(\G(n,d))$ is exponential in $n$ with high probability, the scaling applied in (\ref{eqIsingConv}) is the appropriate one to obtain a finite limit.
Furthermore, by Azuma's inequality $\ln Z_\beta(\G(n,d))$ is concentrated about its expectation.
Therefore, (\ref{eqIsingConv})  implies that $\frac1n\ln Z_\beta(\G(n,d))$ converges to $F(\beta,d)$ in probability.

The \emph{Potts antiferromagnet} on a graph $G=(V,E)$ can be viewed as a twist on the Ising model.
In this case we look at assignments $\sigma:V\to[k]$ for some number $k\geq3$.
The weight functions associated with the edges are defined by
	$\psi_e(\sigma)=\exp(-\beta\vecone\{\sigma(v)=\sigma(w)\})$ for some $\beta>0$.
Thus, this time the edges prefer that the incident vertices receive {\em different} values.
The Gibbs measure and the partition function read
	\begin{align*}
	\mu_{G,\beta}(\sigma)&=\frac1{Z_\beta(G)}\exp\brk{-\beta\sum_{\{v,w\}\in E}\vecone\{\sigma(v)=\sigma(w)\}},&
		Z_\beta(G)&=\sum_{\sigma:V\to[k]}\exp\brk{-\beta\sum_{\{v,w\}\in E}\vecone\{\sigma(v)=\sigma(w)\}}.
	\end{align*}
While it is known that
 	$\lim_{n\to\infty}\frac1n\Erw[\ln Z_\beta(\G(n,d))]$
exists and that $\ln Z_\beta(\G(n,d))$ is concentrated about its expectation~\cite{bayati}, the precise value remains elusive for a wide range of $d,\beta$
	(in contrast ferromagnetic version of the model~\cite{DMSS}).
However, it is not difficult to see that for sufficiently large values of $d,\beta$ we have~\cite{cond}
	$$\lim_{n\to\infty}\frac1n\Erw[\ln Z_\beta(\G(n,d))]<\lim_{n\to\infty}\frac1n\ln \Erw[Z_\beta(\G(n,d))].$$
Hence, just like in the random $k$-SAT model the first moment overshoots the actual value of the partition function by an exponential factor.
The Potts model is closely related to the $k$-colorability problem.
Indeed, if we think of the $k$ possible values as colors, then for large $\beta$
the Gibbs measure concentrates on colorings with few monochromatic edges.

As a third example let us consider the following version of the random $k$-SAT model.
Let $k\geq3$, $\Delta>1$ be fixed integers, let $V_n=\{x_1,\ldots,x_n\}$ be a set of Boolean variables and let $d_n:V_n\times\{\pm1\}\to[\Delta]$ be a map
such that $$m=\sum_{x\in V_n}(d_n(x,1)+d_n(x,-1))/k$$ is an integer.
Then we let $\PHI(n,k,d_n)$ be a random $k$-CNF formula with $m$ clauses in which each variable $x\in V_n$ appears precisely $d_n(x,1)$ times
as a positive literal and precisely $d_n(x,-1)$ times as a negative literal.
As in \Sec~\ref{Sec_intro}, for a clause $a$ and a truth assignment $\sigma:V\to\{0,1\}$ we let
	$\psi_a(\sigma)=\exp(-\beta\vecone\{\sigma\mbox{ violates }a\}).$
Then for a given parameter $\beta>0$ we obtain a Gibbs measure that weighs assignments by the number of clauses that they violate
and a corresponding partition function $Z_{\beta}(\PHI(n,k,d_n))$, cf.\ (\ref{eqkSAT1})--(\ref{eqkSAT2}).
Hence, for given $\beta>0$, $k\geq3$ and degree assignments $(d_n)_n$ the problem 
of determining	$\lim_{n\to\infty}\frac1n\Erw[\ln Z_\beta(\PHI(n,k,d_n))]$ arises.
This question is anything but straightforward even in the special case that $d_n(x,\pm1)=d_0$ is the same for all $x$.
In~\cite{clusters} we show how the results of the present paper can be put to work to tackle this case.

\subsection{Random factor graphs}
The following definition  encompasses a variety of concrete models.

\begin{definition}\label{Def_model}
Let $\Delta>0$ be an integer, let $\Omega,\Theta$ be finite sets and let $\Psi=\cbc{\psi_1,\ldots,\psi_l}$ be a finite set of functions
	$\psi_i:\Omega^{h_i}\to(0,\infty)$ of arity $h_i\in[\Delta]$.
A \bemph{$(\Delta,\Omega,\Psi,\Theta)$-model} $\cM=(V,F,d,t,(\psi_a)_{a\in F})$ consists of
	\begin{description}
	\item[M1] a countable set $V$ of \bemph{variable nodes},
	\item[M2] a countable set $F$ of \bemph{constraint nodes},
	\item[M3] a map $d:V\cup F\to\brk{\Delta}$ such that 
			\begin{equation}\label{eqDegSum}
			\sum_{x\in V}d(x)=\sum_{a\in F}d(a),
			\end{equation}
	\item[M4] a map $t:C_V\cup C_F\to\Theta$, where we let
			\begin{align*}
			C_V&=\bigcup_{x\in V}\cbc{x}\times[d(x)],&
			C_F&=\bigcup_{a\in F}\cbc{a}\times[d(a)],
			\end{align*}
		such that 
			\begin{equation}\label{eqTypeSum}
			\abs{t^{-1}(\theta)\cap C_V}=\abs{t^{-1}(\theta)\cap C_F}\quad\mbox{ for each $\theta\in\Theta$,}
			\end{equation}
	\item[M5] a map $F\to\Psi$, $a\mapsto\psi_a$ such that $\psi_a:\Omega^{d(a)}\to(0,\infty)$ for all $a\in F$.
	\end{description}
The \bemph{size} of the model is $\#\cM=|V|$.
Furthermore, a \bemph{$\cM$-factor graph} is a bijection $G:C_V\to C_F$, $(x,i)\mapsto G(x,i)$ such that $t(G(x,i))=t(x,i)$ for all $(x,i)\in C_V$.
\end{definition}

\noindent
Of course, (\ref{eqDegSum}) and (\ref{eqTypeSum})
require that either both quantities are infinite or both are finite.

The semantics is that $\Delta$ is the maximum degree of a factor graph.
Moreover, $\Omega$ is the set of possible values that the variables of the model range over, e.g., the set $\{\pm1\}$ in the Ising model.
Further, $\Theta$ is a set of ``types''.
For instance, in the random $k$-SAT model the types can be used to specify the signs of the literals.
Additionally, $\Psi$ is a set of possible weight functions.

A model $\cM$ comes with a set $V$ of variable nodes and a set $F$ of contraint nodes.
The degrees of these nodes are prescribed by the map $d$.
Just like in the ``configuration model'' of graphs with a given degree sequence
we create $d(v)$ ``clones'' of each node $v$.
The sets $C_V$, $C_F$ contain the clones of the variable and constraint nodes, respectively.
Further, the map $t$ assigns a type to each ``clone'' of either a constraint or variable node and
 each constraint node $a$ comes with a weight function $\psi_a$.

A $\cM$-factor graph is a type-preserving matching $G$ of the variable and constraint clones.
Let $\cG(\cM)$ be the set of all $\cM$-factor graphs
and write $\G=\G(\cM)$ for a uniformly random sample from $\cG(\cM)$.
Contracting the clones of each node,
we obtain a bipartite (multi-)graph with variable nodes $V$ and  constraint nodes $F$.
We often identify $\G$ with this multi-graph.
For instance, if we speak of the distance of two vertices in $\G$ we mean the length of a shortest path in this multi-graph.

For a clone $(x,i)\in C_V$ we denote by $\partial(G,x,i)=G(x,i)$ the clone that $G$ matches $(x,i)$ to.
Similarly, for $(a,j)\in C_F$ we write $\partial(G,a,j)$ for the variable clone $(x,i)$ such that $\partial(G,x,i)=(a,j)$.
Moreover, for a variable $x$ we let $\partial(G,x)=\{\partial(G,x,i):i\in[d(x)]\}$ and analogously for $a\in F$ we set $\partial(G,a)=\{\partial(G,a,j):j\in[d(a)]\}$.
To economise notation we sometimes identify a clone $(x,i)$ with the underlying variable $x$.
For instance, if $\sigma:V\to\Omega$ is an assignment, then we take the liberty of writing $\sigma(x,i)=\sigma(x)$.
Additionally, where convenient we view $\partial(G,x)$ as the set of all constraint nodes $a\in F$ such that there exist $i\in[d(x)]$, $j\in[d(a)]$ such that $(a,j)=G(x,i)$.
The corresponding convention applies to $\partial(G,a)$.

A \bemph{$\cM$-assignment} is a map $\sigma:V\to\Omega$ 
and  we define
	\begin{align*}
	\psi_{G,a}(\sigma)&=\psi_a\big(\sigma(\partial_G(a,1)),\ldots,\sigma(\partial_G(a,d(a)))\big)\qquad\mbox{for }a\in F,\quad\mbox{and}\quad&
	\psi_G(\sigma)&=\prod_{a\in F}\psi_a(\sigma).
	\end{align*}
Further, the {\bem Gibbs distribution} and the \bemph{partition function} of $G$ are
	\begin{align}\label{eqZ}
	\mu_G(\sigma)&=\psi_G(\sigma)/Z_G,\quad\mbox{where}&
	Z(G)&=\sum_{\sigma:V\to\Omega}\psi_G(\sigma).
	\end{align}
We denote expectations with respect to the Gibbs measure by $\bck{\nix}_G=\bck{\nix}_{\mu_G}$.

The fundamental problem that arises is the study of the random variable $\ln Z(\G)$.
As mentioned in \Sec~\ref{Sec_intro},
this random variable holds the key to getting a handle the Gibbs measure and thus the combinatorics of the problem.
The following proposition establishes concentration about the expectation.
For two factor graphs $G,G'\in\cG(\cM)$ let
	\begin{align}\label{eqDist}
	\dist(G,G')&=\abs{\cbc{(x,i)\in C_V:\partial(G,x,i)\neq\partial(G',x,i)}}.
	\end{align}

\begin{proposition}\label{Lemma_conc}
For any $\Delta,\Omega,\Theta,\Psi$ there exists
$\eta=\eta(\Delta,\Omega,\Theta,\Psi)>0$ such that for any $(\Delta,\Omega,\Psi,\Theta)$-model $\cM$ of size $n=\#\cM\geq1/\eta$ 
and any $\eps>0$ we have
	$\pr\brk{\abs{\ln Z(\G)-\Erw[\ln Z(\G)]}>\eps}\leq\exp(-\eta\eps^2 n)$.
\end{proposition}
\begin{proof}
There exists a number $\rho>0$ that depends on $\Delta,\Omega,\Psi,\Theta$ only such that
for any two factor graphs $G,G'\in\cG(\cM)$ we have $|\ln Z(G)-\ln Z(G')|\leq\rho\cdot\dist(G,G')$.
Therefore,  the assertion follows from Azuma's inequality.
\end{proof}

Thus, \Prop~\ref{Lemma_conc} reduces our task to calculating the expectation $\Erw[\ln Z(\G)]$.
Generally, the standard first and second moment method do not suffice to tackle this problem
because the logarithm sits {\em inside} the expectation.
While, of course, Jensen's inequality guarantees that
	\begin{align}\label{eqLemmaAnnealed}
	\Erw[\ln Z(\G)]&\leq\ln\Erw[Z(\G)],
	\end{align}
equality does not typically hold.
In fact, we saw examples where $\ln\Erw[Z(\G)]-\Erw[\ln Z(\G)]$ is linear in the size $\#\cM$ of the model already.
If so, then the Paley-Zygmund inequality entails that $\ln(\Erw[Z(\G)^2]/\Erw[Z(\G)]^2)$ is linear in $\#\cM$ as well,
dooming the second moment method.
Furthermore, even if $\Erw[\ln Z(\G)]\sim\ln\Erw[Z(\G)]$ the second moment method does not generally succeed~\cite{Lenka}.
Let us now revisit the examples from \Sec~\ref{Sec_FactorGraphExamples}.

\begin{example}[the Ising model on the random $d$-regular graph]\label{Ex_Ising}\upshape
Suppose that $d\geq2,\beta>0$.
Let $\Delta=d$, $\Omega=\cbc{\pm1}$,  $\Psi=\{\psi\}$, where
	$\psi:\{\pm1\}^2\to(0,\infty)$, $(\sigma_1,\sigma_2)\mapsto\exp(\beta\sigma_1\sigma_2)$, and set $\Theta=\{0\}$.
Further, given $n\geq1$ such that $dn$ is even we define a $(\Delta,\Omega,\Psi,\Theta)$-model $\cM(d,n)$
by letting $V=\{x_1,\ldots,x_n\}$, $F=\{a_1,\ldots,a_{dn/2}\}$, $d(x)=d$ for all $x\in V$, $d(a)=2$ for all $a\in F$,
$t(x,i)=t(f,j)=0$ for all $(x,i)\in C_V$, $(f,j)\in C_F$, and $\psi_a=\psi$ for all $a\in F$.
Thus, all clones have the same ``type'' and all constraint nodes have arity two and the same weight function.
Hence, the random graph $\G(\cM)$ is obtained by matching the $dn$ variable clones randomly to the $dn$ constraint clones.
If we simply replace the constraint nodes, which have degree two, by edges joining the two adjacent variable nodes,
then the resulting random multigraph is contiguous to the uniformly random $d$-regular graph on $n$ vertices.
In the model $\cM$ (\ref{eqLemmaAnnealed}) holds with (asymptotic) equality for all $d,\beta$~\cite{DM}.
\end{example}

\begin{example}[the Potts antiferromagnet on the random $d$-regular graph]\label{Ex_Potts}\upshape
The construction is similar to the previous example, except that $\Omega=[k]$ is the set of colors and 
$\psi(\sigma_1,\sigma_2)=\exp(-\beta\vecone\{\sigma_1=\sigma_2\})$.
In this example (\ref{eqLemmaAnnealed}) holds with asymptotic equality if either $d\leq d_0(k)$ or 
	$d>d_0(k)$ and $\beta\leq\beta_0(d,k)$ for certain critical values $d_0(k)$, $\beta_0(d,k)$.
However, for sufficiently large $d,\beta$ there occurs a linear gap~\cite{cond,CDGS}.
\end{example}

\begin{example}[random $k$-SAT]\label{Ex_Potts}\upshape
To capture the random $k$-SAT model we let $\Delta>0$ be a maximum degree  and $\Omega=\Theta=\{\pm1\}$.
Further, each $s\in\{\pm1\}^k$ gives rise to a function
	$$\psi_s:\{\pm1\}^k\to(0,\infty),\qquad \sigma\mapsto\exp(-\beta\vecone\{\sigma=-s\})$$
and we let $\Psi=\{\psi_s:s\in\{\pm1\}^k\}$.
The idea is that $s$ is the ``sign pattern'' of a $k$-clause, with $s_i=\pm1$ indicating that the $i$th literal is positive/negative.
Then a truth assignment $\sigma$ of the $k$ variables is satisfying unless $\sigma_i=-s_i$ for all $i$.
The corresponding model $\cM$ has a set $V=\{x_1,\ldots,x_n\}$ of Boolean variables and a set $F=\{a_1,\ldots,a_m\}$ of clauses.
Moreover, the map $d:V\to[\Delta]$ prescribes the degree of each variable, while of course each clause has degree $k$.
Additionally, the map $t:C_V\cup C_F\to\Theta=\{\pm1\}$ prescribes the positive/negative occurrences of the variables and the sign patterns of the clauses.
Thus, a variable $x$ occurs $|\{i\in[d(v)]:t(x,i)=\pm1\}|$ times positively/negatively and the $j$th literal of a clause $a$ is positive iff $t(a,j)=1$.
Finally, the weight function of clause $a$ is $\psi_{(t(a,1),\ldots,t(a,k))}$.
The bound (\ref{eqLemmaAnnealed}) does not generally hold with equality~\cite{maxsat,clusters}.
\end{example}

While \Def~\ref{Def_model} encompasses many problems of interest, there are two  restrictions.
First, because all weight functions $\psi\in\Psi$ take strictly positive values, \Def~\ref{Def_model} does not allow for ``hard'' constraints.
For instance, \Def~\ref{Def_model} does not accommodate the graph coloring problem, which imposes the strict requirement that no single edge be monochromatic.
However, hard constraints can be approximated by soft ones, e.g., by choosing a very large value of $\beta$ in the Potts antiferromagnet.
Moreover, many of the arguments in the following sections do extend to hard constraints with a bit of care.
However, the assumption that all $\psi$ are strictly positive saves us many case distinctions as
it ensures that $Z(\G)$ is strictly positive  and that therefore the Gibbs measure is well-defined.

The second restriction is that we prescribe a fixed maximum degree $\Delta$.
Thus, if we consider a sequence $\underline\cM=(\cM_n)_n$ of $(\Delta,\Omega,\Psi,\Theta)$-models with $\#\cM_n=n$, then all factor graphs have a bounded degree.
By comparison, if we choose a $k$-SAT formula with $n$ variables and $m=\alpha n/k$ clauses uniformly at random for fixed $k\geq3,\alpha>0$,
then the maximum variable degree will be of order $\ln n/\ln\ln n$.
Yet this case can be approximated well by a sequence of models with a large enough maximum degree $\Delta$.
In fact, if we calculate $\Erw[\ln Z]$ for any fixed $\Delta$, then the $\Delta\to\infty$ limit is easily seen to yield the answer in the case of
uniformly random formulas.
Nevertheless, the bounded degree assumption is 
technically convenient because it facilitates the use of local weak convergence, as we will discuss next.

\begin{remark}
For the sake of simplicity in (\ref{eqZ}) we definied the partition function as the sum over all $\sigma:V\to\Omega$.
However, the results stated in the following carry over to the cases where $Z$ is defined as the sum over all configurations
of a subset of $\emptyset\neq\cC_\cM\subset\Omega^V$, e.g., all $\sigma$ that have Hamming distance at most $\alpha n$ from some reference assignment $\sigma_0$
for a fixed $\alpha>0$.
Of course, in this case the Gibbs measure is defined such that its support is equal to $\cC_\cM$.
\end{remark}

\subsection{Local weak convergence}
Suppose that we fix $\Delta,\Omega,\Psi,\Theta$ as in \Def~\ref{Def_model} and that $\underline\cM=(\cM_n)_n$ is a sequence of $(\Delta,\Omega,\Psi,\Theta)$-models
such that $\cM_n=(V_n,F_n,d_n,t_n,(\psi_a)_{a\in F_n})$ has size $n$.
Let us write $\G=\G(\cM_n)$ for the sake of brevity.
According to the cavity method, $\lim_{n\to\infty}\frac1n\Erw[\ln Z(\G)]$ is determined by the ``limiting local structure'' of the random factor graph $\G$.
To formalise this concept, we adapt the concept of {\em local weak convergence} of graph sequences~\cite[Part~4]{Lovasz} to our current
setup, thereby generalising the approach taken in~\cite{DMS}.

\begin{definition}\label{Def_template_2}
A $(\Delta,\Omega,\Psi,\Theta)$-\bemph{template} consists of
a $(\Delta,\Omega,\Psi,\Theta)$-model $\cM$,
a connected factor graph $H\in\cG(\cM)$ and a \bemph{root} $r_H$, which is a variable  or  factor node.
Its \bemph{size} is $\#\cM$.
Moreover, two templates $H,H'$ 
with models $\cM=(V,F,d,t,(\psi_a))$,
$\cM'=(V',F',d',t',(\psi_a'))$
are \bemph{isomorphic} if there exists a bijection $\pi:V\cup F\to V'\cup F'$ such that
	\begin{description}
	\item[ISM1] $\pi(r_H)=r_H'$,
	\item[ISM2] $\pi(V)=V'$ and $\pi(F)= F'$,
	\item[ISM3] $d(v)=d'(\pi(v))$ for all $v\in V\cup F$,
	\item[ISM4] $t(v,i)=t'(\pi(v),i)$ for all $(v,i)\in C_V\cup C_F$,
	\item[ISM5] $\psi_a=\psi_{\pi(a)}$ for all $a\in F$, and
	\item[ISM6] if $(v,i)\in C_V,(a,j)\in C_F$ satisfy $\partial(G,x,i)=(a,j)$,
		then $\partial(G',\pi(x),i)=(\pi(a),j)$.
	\end{description}
\end{definition}

\noindent
Thus, a template is, basically, a finite or countably infinite connected factor graph with a distinguished root.
Moreover, an isomorphism preserves the root as well as degrees, types, weight functions and adjacencies.

Let us write $[H]$ for the isomorphism class of a template and let $\fG=\fG(\Delta,\Omega,\Theta,\Psi)$ be the set of all isomorphism classes
of $(\Delta,\Omega,\Psi,\Theta)$-templates.
For each $[H]\in\fG$ and $\ell\geq1$ let $\partial^\ell[H]$ be the isomorphism class of the template obtained
by removing all vertices at a distance greater than $\ell$ from the root.
We endow $\fG$ with the coarsest topology that makes all the functions
	$$\Gamma\in\fG\mapsto\vecone\{\partial^\ell[\Gamma]=\partial^\ell[\Gamma_0]\}\in\{0,1\}\qquad\mbox{for $\ell\geq1,\Gamma_0\in\fG$}$$
continuous.
Moreover, the space $\cP(\fG)$ of probability measures on $\fG$ carries the weak topology.
So does the space $\cP^2(\fG)$ of probability measures on $\cP(\fG)$.
For $\Gamma\in\fG$ we write $\atom_\Gamma\in\cP(\fG)$ for the Dirac measure that puts mass one on the single point $\Gamma$.
Similarly, for $\lambda\in\cP(\fG)$ we let $\atom_\lambda\in\cP^2(\fG)$ be the Dirac measure on $\lambda$.
Our assumption that the maximum degree is bounded by a fixed number $\Delta$ ensures that $\fG$, $\cP(\fG)$, $\cP^2(\fG)$ are compact Polish spaces.

For a factor graph $G\in\cG(\cM_n)$ and a variable or constraint node $v$ we write $[G,v]$ for the isomorphism class of the 
connected component of $v$ in $G$ rooted at $v$.
Then each  factor graph $G\in\cG(\cM_n)$ gives rise to the empirical distribution
	$$\lambda_{G}=\frac1{|V_n|+|F_n|}\sum_{v\in V_n\cup F_n}\atom_{[G,v]}\in\cP(\fG).$$
We say that $\underline\cM$ {\bem converges locally} to $\thet\in\cP(\fG)$ if
	\begin{equation}\label{eqLocalWeakConvergence}
	\lim_{n\to\infty}\Erw[\atom_{\lambda_{\G}}]=\atom_\thet.
	\end{equation}
Denote a random isomorphism class chosen from the distribution $\thet$ by $\T=\T_{\thet}$.
Unravelling the definitions, we see that (\ref{eqLocalWeakConvergence}) holds iff for every integer $\ell>0$
and every $[H]\in\fG$ we have
	\begin{align}\label{eqLocalWeakConvergence2}
	\frac1{|V_n|+|F_n|}\sum_{v\in V_n\cup F_n}\vecone\{\partial^\ell[\G,v]=\partial^\ell[H]\}
			&\ \stacksign{$n\to\infty$}\to\ \pr\brk{\partial^\ell\T_\thet=\partial^\ell[H]}
		\quad\mbox{in probability}.
	\end{align}

We are going to be interested in the case that $\underline\cM$ converges locally to a distribution $\thet$ on {\em acyclic} templates.
Thus, let $\fT$ be the set of all acyclic templates.
Further, we write $\cV$ for the set of all templates whose root is a variable node and $\cF$ for the set of all templates whose root is a constraint node.
Additionally, for a template $[H]$ we write $r_{[H]}$ for the root vertex, $d_{[H]}$ for its degree and $\psi_{[H]}$ for the weight function of the root vertex if $[H]\in\cF$.
Moreover, for $j\in[d_{[H]}]$ we write $[H]\reroot j$ for the template obtained from $[H]$ by re-rooting the template at the $j$th neighbor of $r_{[H]}$.
(This makes sense because condition {\bf ISM6} from \Def~\ref{Def_template_2} preserves the order of the neighbors.)

We will frequently condition on the depth-$\ell$ neighborhood of the random factor graph $\G$ for some finite $\ell$.
Hence, for  $G,G'\in\cG(\cM_n)$ and $\ell\geq1$
we write $G\ism_\ell G'$ if $\partial^\ell[G,x]=\partial^\ell[G',x]$ for all variable nodes $x\in V_n$ and $\partial^{\ell+1}[G,a]=\partial^{\ell+1}[G',a]$
for all constraint nodes $a\in F_n$.
Let $\cT_\ell=\cT_{\ell,\cM_n}$ be the $\sigma$-algebra on $\cG(\cM_n)$ generated by the equivalence classes of the relation $\ism_\ell$.
Additionally, for $G\in\cG(\cM_n)$ and $\ell\geq0$ we let
	$$\lambda_{G,\ell}=\frac1{|V_n|+|F_n|}\brk{\sum_{x\in V_n}\atom_{\partial^\ell[G,x]}+\sum_{a\in F_n}\atom_{\partial^{\ell+1}[G,a]}}$$
be the empirical distribution of the depth-$\ell$ neighborhood structure.

Furthermore, let
		$$\fT_\ell=\cbc{\partial^\ell T:T\in\fT\cap\cV}\cup\cbc{\partial^{\ell+1} T:T\in\fT\cap\cF}.$$
Then for a probability measure $\thet\in\cP(\fT)$ we denote by $\thet_\ell$ the image of $\thet$ under
the map
	$$\fT\to\fT_\ell,\qquad T\mapsto\begin{cases}
		\partial^\ell T&\mbox{ if }T\in\fT\cap\cV,\\
		\partial^{\ell+1} T&\mbox{ if }T\in\fT\cap\cF.
		\end{cases}$$
Because all degrees are bounded by $\Delta$, the set $\fT_\ell$ is finite for every $\ell\geq1$.
Hence,  (\ref{eqLocalWeakConvergence2}) entails that 
$\underline\cM$ converges locally to $\thet\in\cP(\fT)$ iff
	\begin{align}\label{eqLocalWeakConvergence3}
	\lim_{n\to\infty}\Erw\TV{\lambda_{\G,\ell}-\thet_\ell}&=0\qquad\mbox{for every }\ell\geq1.
	\end{align}

\subsection{The planted distribution}
While $\G$ is chosen uniformly at random (from the configuration model),
we need to consider another distribution that weighs factor graphs by their partition function. 
Specifically, given $\ell\geq0$ let $\hat\G_\ell=\hat\G_{\ell,\cM_n}$ be a random graph chosen according to the distribution
	\begin{align}\label{eqPlantedDistribution}
	\pr\brk{\hat\G_\ell=G}&=Z(G)\cdot\Erw\brk{\frac{\vecone\{\G=G\}}{\Erw[Z|\cT_\ell]}}\qquad(G\in\cG(\cM_n)),
	\end{align}
which we call the \bemph{planted distribution}.
The definition (\ref{eqPlantedDistribution}) ensures that the distribution of the ``depth-$\ell$ neighborhood structure'' of $\hat\G_\ell$
coincides with that of $\G$.

Perhaps more intuitively, the planted distribution can be described by the following experiment.
First step, choose a random factor graph $\G$.
Then, given $\G$, choose the factor graph $\hat\G_\ell$ randomly such that
a graph $G\ism_\ell\G$ comes up with a probability that is proportional to $Z(G)$.
Perhaps despite appearances, the planted distribution is reasonably easy to work with in many cases.
For instance, it has been employed successfully to study random $k$-SAT as well
as random graph or hypergraph coloring problems~\cite{Barriers,clusters,hyp2col,Lenka,DSS3}.

\subsection{Short cycles}
In most cases of interest the random factor graph is unlikely to contain many short cycles, and it will be convenient for us to exploit this fact.
Hence,  let us call a factor graph $G$ {\bem $l$-acyclic} if it does not contain a cycle of length at most $l$.
We say that the sequence $\underline\cM$ of models has {\bem  high girth} if for any $\ell,l>0$ we have
	\begin{equation}\label{eqShortCycles}
	\liminf_{n\to\infty}\,
		\pr\brk{\mbox{$\G$ is $l$-acyclic}}>0,\qquad
		\liminf_{n\to\infty}\,
		\pr\brk{\mbox{$\hat\G_\ell$ is $l$-acyclic}}>0.
	\end{equation}
Thus, there is a non-vanishing probability that the random factor graph $\G$ is $l$-acyclic.
Moreover, short cycles do not have too heavy an impact on the partition function
as the graph chosen from the planted distribution has a non-vanishing probability of being $l$-acyclic as well.

In the following, we are going to denote the event that a random factor graph is $l$-acyclic by $\cA_l$.
Let us highlight the following consequence of the high girth condition and the construction of the planted distribution.

\begin{proposition}\label{Prop_plantedModel}
Assume that $\underline\cM$ is a sequence of $(\Delta,\Omega,\Psi,\Theta)$-models of high girth.
Let $\ell\geq1$ be an integer and suppose that $\cB$ is an event such that
	$\lim_{n\to\infty}\pr\brk{\hat\G_\ell\in\cB}=1$.
If $b$ is a real and $l\geq0$ is an integer  such that
	\begin{align}\label{eqProp_plantedModelAssm}
	\lim_{n\to\infty}\pr\brk{\ln\Erw[Z(\G)|\cT_\ell]\geq b n|\cA_l}=1,
	\end{align}
then 
	$\lim_{n\to\infty}\frac1n\ln\Erw\brk{\vecone\{\cB\cap\cA_l\}Z(\G)}\geq b$.
\end{proposition}
\begin{proof}
Since $\lim_{n\to\infty}\pr\brk{\hat\G_\ell\in\cB}=1$ the high girth condition (\ref{eqShortCycles}) implies that
	$\lim_{n\to\infty}\pr\brk{\hat\G_\ell\in\cB|\cA_l}=1$
for every $l$.
Set $\cB_l=\cA_l\cap\cB$.
Then 
by the definition~(\ref{eqPlantedDistribution}) of the planted distribution,
	\begin{align*}
	1-o(1)&=\pr\brk{\hat\G_\ell\in\cB|\cA_l}=\sum_{G\in\cB_l}Z(G)\Erw\brk{\frac{\vecone\{\G=G\}}{\Erw[Z|\cT_\ell]}\bigg|\cA_l}
		=\Erw\brk{\frac{\vecone\{\G\in\cB_l\}Z(\G)}{\Erw[Z|\cT_\ell]}\bigg|\cA_l}
		=\Erw\brk{\frac{\Erw[\vecone\{\G\in\cB_l\}Z|\cT_\ell]}{\Erw[Z|\cT_\ell]}\bigg|\cA_l}.
	\end{align*}
Consequently,
	$\pr\brk{\Erw[\vecone\{\G\in\cB_l\}Z]|\cT_\ell]\geq\Erw[Z|\cT_\ell]/2|\cA_l}=1-o(1)$.
Hence, (\ref{eqProp_plantedModelAssm}) yields
	$$\pr\brk{\ln\Erw[\vecone\{\G\in\cB_l\}Z]|\cT_\ell]\geq bn-1|\cA_l}=1-o(1).$$
Therefore, the assertion follows from (\ref{eqShortCycles}).
\end{proof}

\begin{remark}
Strictly speaking, the first condition in (\ref{eqShortCycles}) is superfluous as it is implied by the second one.
\end{remark}

\medskip

\smallskip\noindent
{\em From here on out we assume that $\underline\cM$ is a sequence of $(\Delta,\Omega,\Psi,\Theta)$-models of high girth that converges locally to $\thet\in\cP(\fT)$ and
	we fix $\Delta,\Omega,\Theta,\Psi$ for the rest of the paper.}

\section{The Bethe free energy}\label{Sec_Bethe}
\noindent
In this section we present the main results of the paper.
The thrust is that certain basic properties of the Gibbs measure entail an asymptotic formula for $\Erw[\ln Z(\G)]$.
The results are guided by the physics predictions from~\cite{pnas}.

\subsection{An educated guess}

The formula for $\Erw[\ln Z(\G)]$ that the cavity method predicts, the so-called ``replica symmetric solution'',
comes in terms of the distribution $\thet$ to which $\underline\cM$ converges locally.
Thus, the cavity method claims that in order to calculate $\Erw[\ln Z(\G)]$ it is not necessary to deal with the mind-boggling
complexity of the random factor graph with its expansion properties, long cycles etc.
Instead, it suffices to think about the random tree $\T=\T_\thet$, a dramatically simpler object.
The following definition will help us formalise this notion.

\begin{definition}\label{Def_margAssign}
A {\bem marginal assignment} is a measurable map $p:\fT\to\bigcup_{j=1}^\Delta\cP(\Omega^j)$, $T\mapsto p_T$ such that
	\begin{description}
	\item[MA1] $p_T\in\cP(\Omega)$ for all $T\in\cV$,
	\item[MA2] $p_T\in\cP(\Omega^{d_T})$
		 and $p_{T\marg j}=p_{T\reroot j}$ for all $T\in\cF,j\in[d_T]$,
	\item[MA3] For all $T\in\cF$ we have
		\begin{align}\label{eqConstraintMargs}
		H(p_T)+\bck{\ln\psi_T(\SIGMA)}_{p_T}&=\max\cbc{
			H(\nu)+\bck{\ln\psi_T(\SIGMA)}_{\nu}:\nu\in\cP(\Omega^{d_T})\mbox{ s.t.\ }\nu_{\marg j}=p_{T\reroot j}\mbox{ for all }j\in[d_T]}.
		\end{align}
	\end{description}
Further, the {\bem Bethe free energy} of  $p$ with respect to $\thet$ is
	\begin{align}\label{eqBetheFreeEnergy}
	\cB_\thet(p)&=
		\Erw\brk{(1-d_{\T})H( p_{\T})|\cV}+\frac{\pr\brk{\T\in\cF}}{\pr\brk{\T\in\cV}}\Erw \brk{H(p_{\T})+\bck{\ln\psi_{\T}(\SIGMA)}_{p_{\T}}|\cF},
	\end{align}
where, of course, $\Erw[\nix],\pr[\nix]$ refer to the choice of the random tree $\T=\T_\thet$.
\end{definition}

Thus, a marginal assignment provides a probability distribution $p_T$ on $\Omega$ for each tree whose root is a variable node.
Furthermore, for trees $T$ rooted at a contraint node $p_T$ is a distribution on $\Omega^{d_T}$, which we think of as the joint distribution
of the variables involved in the constraint.
The distributions assigned to $T$ rooted at a constraint node must satisfy a consistency condition:
	the $j$th marginal of $p_T$ has to coincide with the distribution assigned to the tree $T\reroot j$ rooted at the $j$th child
	of the root of $T$ for every $j\in[d_T]$;
of course, $T\reroot j$ is a tree rooted at a variable node.
In addition, {\bf MA3} requires that for $T\in\cF$ the distribution $p_T$ maximises the functional $H(\nu)+\bck{\psi_T(\SIGMA)}_{\nu}$
amongst all distribution $\nu$ with the same marginal distributions as $p_T$.
Furthermore, the Bethe free energy is a functional that maps each marginal assignment $p$ to a real number.
For a detailed derivation of this formula based on physics intuition we refer to~\cite{MM}.

Given a distribution $\thet$ on trees, the cavity method provides a plausible recipe for constructing marginal assignments.
Roughly speaking, the idea is to identify fixed points of an operator called Belief Propagation on the random infinite tree~\cite{MM}.
However, this procedure is difficult to formalise mathematically because generally there are several Belief Propagation fixed points
and model-dependent considerations are necessary to identify the ``correct'' one.
To keep matters as simple as possible we are therefore going to assume that a marginal assignment is given.

\begin{remark}
Because the entropy is concave, conditions {\bf MA2} and {\bf MA3} specify the distributions $p_T$ for $T\in\cF$ uniquely.
In other words, a marginal assignment is actually determined completely by the distributions $p_T$ for $T\in\cV$.
\end{remark}

For a marginal assignment $p$, an integer $\ell$ and a tree $T\in\fT_\ell\cap\cV$ we define
	$$p_{\ell,T}=\Erw[p_{\T}|\partial^\ell\T=T].$$
Thus, $p_{\ell}$ is the conditional expectation of $p$ given the first $\ell$ layers of the tree.
Finally, to avoid notational hazards we let $p_T,p_{\ell,T}$ be the  uniform distribution on $\Omega$ for all $T\in\fG\setminus\cT$.

\begin{lemma}\label{Lemma_martingale}
For any $\eps>0$ there is $\ell_0>0$ such that for all $\ell>\ell_0$ we have
	$\Erw[\TV{p_{\ell,\partial^\ell\T}-p_{\T}}|\T\in\cV]<\eps.$
\end{lemma}
\begin{proof}
Define an equivalence relation $\equiv_\ell$ on $\fT\cap\cV$ by letting
$T\equiv_\ell T'$ iff $\partial^\ell T=\partial^\ell T'$.
Then for any $\omega\in\Omega$ the sequence of random variables
	$X_\ell(\T)=p_{\ell,\partial^\ell\T}(\omega)$
is a martingale with respect to the filtration generated by the equivalence classes of $\equiv_\ell$.
By the martingale convergence theorem~\cite[\Thm~5.7]{Durrett}, $(p_\ell)_\ell$ converges $\thet$-almost surely to $p$.
\end{proof}

\subsection{Symmetry}

In the terminology of \Sec~\ref{Sec_hom}, the cavity method claims that $\frac1n\Erw[\ln Z(\G)]$ converges to
 the Bethe free energy of a suitable marginal assignment iff
	\begin{equation}\label{eqNonCondensation}
	\lim_{n\to\infty}\pr\brk{\mu_{\G}\mbox{ is $(\eps,2)$-symmetric}}=1\qquad\mbox{for any }\eps>0\mbox{ (see \cite{pnas})}.
	\end{equation}
This claim is, of course, based on bold non-rigorous deliberations.
Nonetheless, we aim to prove a rigorous statement that comes reasonably close.

To this end, let $p$ be a marginal assignment.
We say that $\underline\cM$ is {\bem $p$-symmetric} if for every $\eps>0$
there is $\ell_0>0$ such that for all $\ell>\ell_0$ we have
	\begin{equation}\label{eqMyNonCondensation}
	\lim_{n\to\infty}\ \pr\brk{\frac1{n^2}\sum_{x,y\in V_n}\TV{\mu_{\G\marg\{x,y\}}-p_{\ell,\partial^\ell[\G,x]}\tensor p_{\ell,\partial^\ell[\G,y]}}>\eps}=0.
	\end{equation}
In other words, for any $\eps>0$ for $\ell$ sufficiently large random factor graph $\G$ enjoys the following property with high probability.
If we pick two variable nodes $x,y$ of $\G$ uniformly and independently, then the joint distribution $\mu_{\G\marg\{x,y\}}$ 
is close to the product distribution $p_{\ell,\partial^\ell[\G,x]}\tensor p_{\ell,\partial^\ell[\G,y]}$ determined by the depth-$\ell$ neighborhoods of $x,y$.
Of course, as $\G$ has bounded maximum degree the distance between randomly chosen $x,y$ is going to be greater than, say, $\ln\ln n$ with high probability.
Thus, similar in spirit to (\ref{eqNonCondensation}),
(\ref{eqMyNonCondensation}) provided that far-apart variables typically decorrelate and that $p$ captures the Gibbs marginals.

In analogy to (\ref{eqMyNonCondensation}), we say that the {\bem planted distribution of $\underline\cM$ is $p$-symmetric} if 
for every $\eps>0$ there is $\ell_0>0$ such that for all $\ell>\ell_0$ we have
	$$\lim_{n\to\infty}\ 
		\pr\brk{\frac1{n^2}\sum_{x,y\in V_n}\TV{\mu_{\hat\G_\ell\marg\{x,y\}}-p_{\ell,\partial^\ell[\hat\G_\ell,x]}\tensor p_{\ell,\partial^\ell[\hat\G_\ell,y]}}>\eps}=0
			\qquad\mbox{for any $\eps>0$}.$$
The main result of this paper is

\begin{theorem}\label{Thm_symUpperBound}
If $\underline\cM$ is $p$-symmetric, then
	$$\limsup_{n\to\infty}\frac 1n\Erw[\ln Z(\G)]\leq\cB_\thet(p).$$
If the planted distribution of $\underline\cM$ is $p$-symmetric as well, then
	$$\lim_{n\to\infty}\frac 1n\Erw[\ln Z(\G)]=\cB_\thet(p).$$
\end{theorem}

Thus, the basic symmetry assumption (\ref{eqMyNonCondensation}) implies that $\cB_\thet(p)$ is an upper bound on $\frac1n\Erw[\ln Z(\G)]$.
If, additionally, the symmetry condition holds in the planted model, then this upper bound is tight.
In particular, in this case $\frac1n\Erw[\ln Z(\G)]$ is completely determined by the limiting local structure $\thet$ and $p$.

The proof of \Thm~\ref{Thm_symUpperBound}, which can be found in \Sec~\ref{Sec_symUpperBound}, is based on \Thm~\ref{Thm_decomp},
	the decomposition theorem for probability measures on cubes.
More precisely, we combine \Thm~\ref{Thm_decomp} with a conditional first and a second moment argument given the local structure of the factor graph,
i.e., given $\cT_\ell$ for a large $\ell$.
The fact that it is necessary to condition on the local structure in order to cope with ``lottery effects'' has been noticed in prior work~\cite{yuval,kSAT,DM,DMS}.
Most prominently, such a conditioning was crucial in order to obtain the precise $k$-SAT threshold for large enough $k$~\cite{DSS3}.
But here the key insight is that \Thm~\ref{Thm_decomp}  enables us to carry out conditional moment calculations in a fairly elegant and generic way.
	
The obvious question that arises from \Thm~\ref{Thm_symUpperBound} is whether there is a simple way to show that
$\underline\cM$ is $p$-symmetric (and that the same is true of the planted distribution).
In Sections~\ref{Sec_nonRe} and~\ref{Sec_GU} we provide two sufficient conditions called non-reconstruction and Gibbs uniqueness.
That these two conditions entail symmetry was predicted in~\cite{pnas}, and \Thm~\ref{Thm_decomp} enables us to prove it.

\subsection{Non-reconstruction}\label{Sec_nonRe}
Following~\cite{pnas} we define a correlation decay condition, the ``non-reconstruction'' condition, on factor graphs and show that it implies symmetry.
The basic idea is to formalise the following.
Given $\eps>0$ pick a large $\ell=\ell(\eps)>1$, choose a random factor graph $\G$ for some large $n$ and  pick a variable node $x$ uniformly at random.
Further, sample an assignment $\SIGMA$ randomly from the Gibbs measure $\mu_{\G}$.
Now, sample a second assignment $\TAU$ from $\mu_{\G}$ subject to the condition that $\TAU(y)=\SIGMA(y)$
for all variable nodes $y$ at distance at least $\ell$ from $x$.
Then non-reconstruction condition asks whether the distribution of $\TAU(x)$ is markedly different from the unconditional marginal $\mu_{\G\marg x}$.
More precisely, non-reconstruction occurs if for any $\eps$ there is $\ell(\eps)$ such that
with high probability $\G$ is such that the shift that a random ``bounary condition'' $\SIGMA$
induces does not exceed $\eps$ in total variation distance.

Of course, instead of conditioning on the values of {\em all} variables at distance at least $\ell$ from $x$,
we might as well just condition on the variables at distance either $\ell$  or $\ell+1$ from $x$, depending on the parity of $\ell$.
This is immediate from the definition (\ref{eqZ}) of the Gibbs measure.

As for the formal definition,  suppose that $G\in\cG(\cM_n)$ is a  factor graph, let $x\in V_n$ and let $\ell\geq1$.
Let $\nabla_\ell(G,x)$ signify the $\sigma$-algebra on $\Omega^n$ generated by the events
	$\vecone\{\SIGMA(y)=\omega\}$ for $\omega\in\Omega$ and $y\in V_n$ at distance either $\ell$ or $\ell+1$ from $x$.
Thus, $\nabla_\ell(G,x)$ pins down all $\SIGMA(y)$ for $y$ at distance $\ell$ from $x$ if $\ell$ is even and $\ell+1$ otherwise.
Then we say that $\underline\cM$ has \bemph{non-reconstruction} with respect to a marginal assignment $p$ if
for any $\eps>0$ there is $\ell>0$ such that
	\begin{align*}
	\lim_{n\to\infty}\pr\brk{\frac1n\sum_{x\in V_n}\bck{\TV{\bck{\TAU[\nix|x]\big|\nabla_\ell(\G,x)}_{\G}-p_{\ell,\partial^\ell[\G,x]}}}_{\G}>\eps}=0.
	\end{align*}
To parse the above, the outer $\pr\brk\nix$ refers to the choice of $\G$.
The big $\bck{\nix}_{\G}$ is the choice of the boundary condition called $\SIGMA$ above.
Finally, $\bck{\nix|\nabla_\ell(\G,x)}_{\G}$ is the random choice given the boundary condition.

Analogously, \bemph{the planted distribution of $\underline\cM$ has non-reconstruction} with respect to $p$ if for any $\eps>0$ there exists $\ell>0$ such that
	\begin{align*}
	\lim_{n\to\infty}\pr\brk{\frac1n\sum_{x\in V_n}\bck{\TV{\bck{\SIGMA[\nix|x]\big|\nabla_\ell(\hat\G_\ell,x)}_{\hat\G_\ell}-
		p_{\ell,\partial^\ell[\hat\G_\ell,x]}}}_{\hat\G_\ell}>\eps}&=0.
	\end{align*}

\begin{theorem}\label{Thm_nonReconstruction}
If $\underline\cM$ has non-reconstruction with respect to $p$, then $\underline\cM$ is $p$-symmetric.
If the planted distribution of $\underline\cM$ has non-reconstruction with respect to $p$, then it is $p$-symmetric.
\end{theorem}

In concrete applications the non-reconstruction condition is typically reasonably easy to verify.
For instance, in~\cite{clusters} we determine the precise location of the so-called ``condensation phase transition''
in the regular $k$-SAT model via \Thm s~\ref{Thm_symUpperBound} and~\ref{Thm_nonReconstruction}.
The proof of \Thm~\ref{Thm_nonReconstruction} can be found in \Sec~\ref{Sec_Thm_nonReconstruction}.

\subsection{Gibbs uniqueness}\label{Sec_GU}
Although the non-reconstruction condition is reasonably handy,  to verify it we still need to ``touch'' the complex random graph $\G$.
Ideally, we might hope for a condition that can be stated solely in terms of the limiting distribution $\thet$ on trees,
	which is conceptually far more accessible.
The ``Gibbs uniqueness'' condition as put forward in~\cite{pnas} fills this order.

Specifically, 
suppose that $T$ is a finite acyclic template  
whose root $r_T$ is a variable node.
Then we say that $T$ is \bemph{$(\eps,\ell)$-unique} with respect to a marginal assignment $p$ if 
	\begin{equation}\label{eqGibbsUniquenessCondition}
	\TV{\bck{\SIGMA[\nix|r_T]\big|\nabla_\ell T}_{T}-p_{T}}<\eps.
	\end{equation}
To parse (\ref{eqGibbsUniquenessCondition}), we observe that $\bck{\SIGMA[\nix|r_T]\big|\nabla_\ell T}_{T}$ is a random variable,
namely the average of the value $\SIGMA[\nix|r_T]$ assigned to the root variable under the Gibbs measure $\mu_T$ given the values of the variables at
distance at least $\ell$ from $r_T$.
Hence, (\ref{eqGibbsUniquenessCondition}) requires that $\bck{\SIGMA[\nix|r_T]\big|\nabla_\ell T}_{T}$ is at total variation distance less than $\eps$
for {\em every} possible assignment of the variables at distance at least $\ell$ from $r_T$, i.e., for every ``boundary condition''.

More generally, we say that $T\in\fT\cap\cV$ is $(\eps,\ell)$-unique with respect to $p$ if the finite template $\partial^{\ell+1}T$ has this property.
(That $\partial^{\ell+1}T$ is finite follows once more from the fact that all degrees are bounded by $\Delta$.)
Further, we call the measure
$\thet\in\cP(\fT)$ {\bem Gibbs-unique} with respect to $p$ if for any $\eps>0$ we have
	$$\lim_{\ell\to\infty}\pr\brk{\T\mbox{ is $(\eps,\ell)$-unique w.r.t.\ }p}=1.$$

\begin{corollary}\label{Thm_smm}
If $\thet\in\cP(\fT)$ is Gibbs-unique with respect to $p$, then
 $\lim_{n\to\infty}\frac1n\Erw[\ln Z(\G)]=\cB_\thet(p)$.
\end{corollary}
\begin{proof}
If $\thet$ is Gibbs-unique with respect to $p$, then (\ref{eqLocalWeakConvergence3}) guarantees that $\underline\cM$ has
non-reconstruction with respect to $p$. Indeed, given $\eps>0, \ell >0$ and a graph $G$ let $\mathcal{E}(G,\eps,\ell)$ denote the set of vertices $x \in V_n$ for which $\partial^\ell[G,x]$ is acyclic and $(\eps,\ell)$ unique. Then we have
\begin{align*} \frac1n\sum_{x\in V_n}\bck{\TV{\bck{\SIGMA[\nix|x]\big|\nabla_\ell(\G,x)}_{\G}-p_{\ell,\partial^\ell[\G,x]}}}_{\G} &\leq  \frac1n\sum_{x\in V_n} \left \| \TV{\bck{\SIGMA[\nix|x]\big|\nabla_\ell(\G,x)}_{\G}-p_{\ell,\partial^\ell[\G,x]}} \right\|_\infty 
\\ & \leq  \eps + \left( 1 - \frac{|\mathcal{E}(\vec G,\eps,\ell)|}{n} \right) , \end{align*}
and by (\ref{eqLocalWeakConvergence3}) $\pr \brk{|\mathcal{E}(\vec G,\eps,\ell)| \leq (1-\eps)n}$ tends to $0$ as $n \to \infty$.
Similarly, because the distribution of the depth-$\ell$ neighborhood structure in the planted distribution $\hat\G_\ell$ coincides
with $\thet_\ell$,  Gibbs-uniqueness implies that the planted model has non-reconstruction with respect to $p$ as well.
Therefore, the assertion follows from
\Thm s~\ref{Thm_symUpperBound} and~\ref{Thm_nonReconstruction}.
\end{proof}

In problems such as the random $k$-SAT model, the Ising model or the Potts antiferromagnet that come with an ``inverse temperature'' parameter $\beta\geq0$,
Gibbs uniqueness is always satisfied for sufficiently small values of $\beta$.
Consequently, \Cor~\ref{Thm_smm} shows that the cavity method always yields the correct value of  $\lim_{n\to\infty}\frac1n\Erw[\ln Z(\G)]$
in the case of small $\beta$, the so-called ``high temperature'' case in physics jargon.
Furthermore, if the Gibbs uniqueness condition is satisfied then there is a canonical way of constructing the marginal assignment $p$
by means of the Belief Propagation algorithm~\cite[\Chap~14]{MM}.
Hence, \Cor~\ref{Thm_smm} provides a full comprehensive answer in this case.

\subsection{Meet the expectation}\label{Sec_expectations}
We proceed to prove \Thm s~\ref{Thm_symUpperBound}. 
To this end, we need to get a handle on the conditional expectation of $Z$ given $\cT_\ell$
and for this purpose we need to study the possible empirical distributions of the values assigned to the variables
of a concrete factor graph $G\in\cG(\cM_n)$.
Specifically, by a {\bem $(G,\ell)$-marginal sequence} we mean a map $q:\fT_\ell\to\bigcup_{j=1}^\Delta\cP(\Omega^j)$, $T\mapsto q_T$ such that
	\begin{description}
	\item[MS1] $q_T\in\cP(\Omega)$ if $T\in\cV\cap\fT_\ell$,
	\item[MS2] $q_T\in\cP(\Omega^{d_T})$ if $T\in\cF\cap\fT_\ell$,
	\item[MS3] for all $T\in\fT_\ell\cap\cV$ we have
		\begin{align}\label{eqMargsWorkOut}
		\sum_{T'\in\fT_\ell\cap\cF}\sum_{j\in[d_{T'}]}\lambda_{G,\ell}(T')\vecone\{\partial^\ell[T'\reroot j]=T\}(q_{T'\marg j}-q_T)&=0.
		\end{align}
	\end{description}
Thus, $q$ assigns each tree $T\in\fT_\ell$ rooted at a variable node a distribution on $\Omega$ and each tree $T\in\fT_\ell$ rooted at a constraint node
a distribution on $\Omega^{d_T}$, just like in \Def~\ref{Def_margAssign}.
Furthermore, the consistency condition (\ref{eqMargsWorkOut}) provides that for a given $T$ rooted at a variable the average
marginal distribution over all $T',j$ such that $\partial^\ell[T'\reroot j]=T$ is equal to $q_T$.
However, in contrast to condition {\bf MA2} from \Def~\ref{Def_margAssign} {\bf MS3} does not require
this marginalisation to work out for every $T',j$ individually.

Suppose now that $U\subset F_n$ is a set of constraint nodes such that $d(a)=d_0$ for all $a\in U$.
Then for $\sigma:V_n\to\Omega$ we let
	\begin{align*}
	\sigma[(\omega_1,\ldots,\omega_{d_0})|U]&=\frac1{|U|}\sum_{a\in U}\prod_{j=1}^{d_0}\vecone\{\sigma(\partial(G,a,j))=\omega_j)\}.
	\end{align*}
Thus, $\sigma[\nix|U]\in\cP(\Omega^{d_0})$ is the empirical distribution of the sequences
	$\{(\sigma(\partial(G,a,1)),\ldots,\sigma(\partial(G,a,d_0))):a\in U\}$.
A factor graph $G$ and $\sigma:V_n\to\Omega$ induce a $(G,\ell)$-marginal sequence $q_{G,\sigma,\ell}$ canonically, namely the empirical distributions
	\begin{align*}
	q_{G,\sigma,\ell,T}&={\sigma[\nix|\{x\in V_n:\partial^\ell[G,x]=T]}&&\mbox{for }T\in\cT_\ell\cap\cV,\\
	q_{G,\sigma,\ell,T}&={\sigma[\nix|\{a\in F_n:\partial^{\ell+1}[G,a]=T\}]}&&\mbox{for }T\in\cT_\ell\cap\cF.
	\end{align*}
Conversely, given a $(G,\ell)$-marginal sequence $q$ let $\Sigma(G,\ell,q,\delta)$ be the set of all $\sigma:V_n\to\Omega$ such that
for all $T\in\fT_\ell\cap\cV$, $T'\in\fT_\ell\cap\cF$ we have
	\begin{align}\label{eqZellqdelta}
	\TV{q_{G,\sigma,\ell,T}-q_T}&\leq \delta,&
	\TV{q_{G,\sigma,\ell,T'}-q_{T'}}&\leq \delta.
	\end{align}
Moreover, let
	\begin{align*}
	Z_{\ell,q,\delta}(G)&=Z(G)\bck{\vecone\{\SIGMA\in\Sigma(G,\ell,q,\delta)\}}_G.
	\end{align*}
Finally, define
	\begin{align*}
	\cB_{G,\ell}(q)&=\sum_{T\in\fT_\ell\cap\cV}(1-d_T)H(q_T)\lambda_{G,\ell}(T|\cV)
		+\frac{|F_n|}{|V_n|}\sum_{T\in\cT_\ell\cap\cF}
			\brk{H(q_T)+\bck{\ln\psi_T(\SIGMA)}_{q_T}
			-\KL{q_{T}}{\bigotimes_{j\in[d_T]}q_{\partial^\ell[T\reroot j]}}}\lambda_{G,\ell}(T|\cF).
	\end{align*}

\noindent
In \Sec~\ref{Sec_crazyConfigs} we are going to prove
the following formula for the expectation of $Z_{\ell,q,\delta}(G)$.

\begin{proposition}\label{Lemma_fmCalc}
For any $\eps>0$, $\ell>0$ there is $\delta>0$ such that for large enough $n$ the following is true.
Assume that $G\in\cG(\cM_n)$ is $100\ell$-acyclic and let $q$ be a $(G,\ell)$-marginal sequence.
Then
	\begin{align*}
	\abs{n^{-1}\ln\Erw[\vecone\{\cA_{2\ell+5}\}Z_{\ell,q,\delta}(\G)|\G\ism_\ell G]-\cB_{G,\ell}(q)}&<\eps.
	\end{align*}
\end{proposition}

We are going to be particularly interested in the expectation of $Z_{\ell,q,\delta}(\G)$ for $q$ ``close'' to a specific marginal assignment $p$.
Formally, a $(G,\ell)$-marginal sequence $q$ is {\bem $(\eps,\ell)$-judicious} with respect to $p$ if
	\begin{align*}
	\sum_{T\in\fT_\ell\cap\cV}\lambda_{G,\ell}[T|\cV]\TV{q_T-p_T}+
	\sum_{T\in\fT_\ell\cap\cF}\sum_{j\in[d_T]}\lambda_{G,\ell}[T|\cF]\TV{q_{T\marg j}-p_{\ell,\partial^\ell[T\reroot j]}}&<\eps.
	\end{align*}
We say that $(G,\sigma)$ is {\bem $(\eps,\ell)$-judicious} with respect to $p$ if the empirical distribution $q_{G,\sigma,\ell}$ is $(\eps,\ell)$-judicious w.r.t.\ $p$.

\begin{corollary}\label{Cor_fmCalc}
For any $\alpha>0$ there exist $\eps>0,\ell>0$ such that for all $0<\beta,\gamma<\eps$ and all $l\geq\ell$ the following is true.
Let $\cL(\gamma,l)$ be the event that $\TV{\lambda_{\G,l}-\thet_l}<\gamma$.
Then
	\begin{align*}
	\limsup_{n\to\infty}\frac1n\ln\Erw\brk{\vecone\{\G\in\cL(\gamma,l)\cap\cA_{100l}\}Z(\G)\bck{(\G,\SIGMA)\mbox{ is $(\beta,l)$-judicious w.r.t.\ $p$}}_{\G}}
		&\leq\cB_\thet(p)+\alpha.
	\end{align*}
\end{corollary}
\begin{proof}
Pick a small enough $\eps=\eps(\alpha)>0$.
By \Lem~\ref{Lemma_martingale} there exists $\ell$ such that 
	$\Erw[\TV{p_{l,\partial^l\T}-p_{\T}}|\cV]<\eps$ \mbox{for all $l\geq\ell$.}
Now, fix any $0<\beta,\gamma<\eps$, $l\geq\ell$, pick $\xi=\xi(\beta,l)$ small enough and assume that $n$ is big enough.
Let $Q(G)$ be the set of all $(G,l)$-marginal sequences that are $(\beta,l)$-judicious w.r.t.\ $p$.
Because $\fT_l$ is a finite set, there exists a number $N=N(\xi)$
such that for every factor graph $G$ there is a subset $Q_*(G)\subset Q(G)$ of size $|Q_*(G)|\leq N$ such that the following is true.
If $(G,\sigma)$ is $(\beta,l)$-judicious w.r.t.\ $p$, then $\sigma\in\bigcup_{q\in Q_*(G)}\Sigma(G,l,q,\xi)$.
Therefore, for all $G$ we have
	\begin{align}\label{eqCor_fmCalc1}
	Z(G)\vecone\{\mbox{$(G,\sigma)$ is $(\eps,l)$-judicious w.r.t.\ $p$}\}&\leq
		N\max_{q\in Q(G)}Z_{\ell,q,\xi}(G).
	\end{align}
\Prop~\ref{Lemma_fmCalc} and (\ref{eqCor_fmCalc1})  imply that
for $\xi$ small enough and $n$ large enough for any factor graph $G\in\cA_{100\ell}$ there is $q^G\in Q(G)$ such that
	\begin{align}\label{eqCor_fmCalc2}
	\ln\Erw[\vecone\{\cA_{100\ell}\}Z(\G)\bck{\vecone\{\mbox{$(\G,\SIGMA)$ is $(\eps,l)$-judicious w.r.t.\ $p$}\}}_{\G}|\G\ism_\ell G]
		&\leq\cB_{G,l}(q^G)+\alpha n/2.
	\end{align}
To proceed, we recall that the Kullback-Leibler divergence is non-negative.
Hence, (\ref{eqCor_fmCalc2}) implies that for large $n$,
	\begin{align}\nonumber
	&\ln\Erw[\vecone\{\cA_{100\ell}\}Z(\G)\bck{\vecone\{\mbox{$(\G,\SIGMA)$ is $(\eps,l)$-judicious w.r.t.\ $p$}\}}_{\G}|\G\ism_\ell G]\\
		&\qquad\leq\alpha n/2+
		\sum_{T\in\fT_\ell\cap\cV}(1-d_T)H(q_T^G)\lambda_G(T|\cV)
		+\frac{|F_n|}{|V_n|}\sum_{T\in\cT_\ell\cap\cF}
			\brk{H(q_T^G)+\bck{\ln\psi_T(\SIGMA)}_{q_T^G}
			}\lambda_G(T|\cF).
			\label{eqCor_fmCalc3}
	\end{align}
Further, for any $j\in[\Delta]$ the function $\nu\in\cP(\Omega^j)\mapsto H(\nu)$ is uniformly continuous
because $\cP(\Omega^j)$ is compact.
By the same token, $\nu\mapsto\bck{\ln\psi(\SIGMA)}_{\nu}$ is uniformly continuous for any $\psi\in\Psi$.
Consequently,
if $G\in\cL(\gamma,l)$ for some $\gamma<\eps$ and $\eps$ is chosen small enough, then (\ref{eqCor_fmCalc3}) entails
	\begin{align}
	\ln\Erw[\vecone\{\cA_{100\ell}\}Z(\G)\bck{\vecone\{\mbox{$(\G,\SIGMA)$ is $(\eps,l)$-judicious w.r.t.\ $p$}\}}_{\G}|\G\ism_\ell G]
		\leq\cB_\thet(p)+\alpha n.
			\label{eqCor_fmCalc4}
	\end{align}
Finallt, the assertion follows from (\ref{eqCor_fmCalc4}) and Bayes' rule.
\end{proof}

\begin{corollary}\label{Cor_fmCalc_lower}
For any $\alpha>0$ there exists $\ell>0$ such that for all $l\geq \ell$ we have
	\begin{align*} 
	\lim_{n\to\infty}\pr\brk{
		\frac1n\ln\Erw\brk{Z(\G)|\cT_l}\leq\cB_\thet(p)-\alpha\bigg|\cA_{100 l}}=0.
	\end{align*}
\end{corollary}
\begin{proof}
Choose a small $\eps=\eps(\alpha)>0$.
By \Lem~\ref{Lemma_martingale} there exists $\ell$ such that  $\Erw[\TV{p_{l,\partial^l\T}-p_{\T}}|\cV]<\eps$ for all $l\geq \ell$.
Hence, fix some $l\geq\ell$ and define
	$q:T\in\fT_l\cap\cV\to\cP(\omega)$, $T\mapsto p_{l,\partial^{l} T}$.
Moreover, for $T\in\fT_l\cap\cF$ let $q_{T}\in\cP(\Omega^{d_T})$ be such that
	$H(q_{T})+\bck{\ln\psi_T(\SIGMA)}_{q_{T}}$ is maximum subject to the condition that  $q_{T\marg j}=q_{\partial^lT\reroot j}$ for all $j\in[d_T]$
(cf.\ (\ref{eqConstraintMargs})).
Further, pick $\delta=\delta(\eps,l)>0$ small enough.
Then \Prop~\ref{Lemma_fmCalc} implies that for large $n$ and any $G\in\cA_{100 l}$
	\begin{align}\nonumber
	&\ln\Erw[Z_{l,q,\delta}(\G)|\G\ism_\ell G]\geq\cB_{G,\ell}(q)-\alpha n/2\\
		&\qquad=-\alpha n/2+\sum_{T\in\fT_\ell\cap\cV}(1-d_T)H(q_T)\lambda_{G,l}(T|\cV)
		+\frac{|F_n|}{|V_n|}\sum_{T\in\cT_\ell\cap\cF}
			\brk{H(q_T)+\bck{\ln\psi_T(\SIGMA)}_{q_T}}\lambda_{G,l}(T|\cF),
				\label{eqCor_fmCalc_lower1}
	\end{align}
because the definition of $q$ ensures that the Kullback-Leibler divergences vanish.
Since $\TV{\thet_l-\lambda_{\G,l}}<\eps$ with high probability by (\ref{eqLocalWeakConvergence3})
and $\Erw[\TV{p_{l,\partial^l\T}-p_{\T}}|\cV]<\eps$, the assertion follows from (\ref{eqCor_fmCalc_lower1}).
\end{proof}

\subsection{Proof of \Thm~\ref{Thm_symUpperBound}}\label{Sec_symUpperBound}
We begin by spelling out the following consequence of the symmetry assumption.

\begin{lemma}\label{Lemma_psymmetric}
If $\underline\cM$ is $p$-symmetric, then
 for any $\eps>0$ for all sufficiently large $\ell$ we have
 	\begin{align}\label{eqLemma_psymmetric_1}
	\lim_{n\to\infty}\pr\brk{\sum_{x\in V_n}\TV{\mu_{\G\marg x}-p_{\ell,\partial^\ell[\G,x]}}>\eps n}&
		=\lim_{n\to\infty}\pr\brk{\mu_{\G}\mbox{ fails to be $(\eps,2)$-symmetric}}=0\quad\mbox{and}\\
	\lim_{n\to\infty}\pr\brk{\sum_{x\in V_n}\TV{\mu_{\hat\G_\ell\marg x}-p_{\ell,\partial^\ell[\hat\G_\ell,x]}}>\eps n}
		&=\lim_{n\to\infty}\pr\brk{\mu_{\hat\G_\ell}\mbox{ fails to be $(\eps,2)$-symmetric}}=0.\label{eqLemma_psymmetric_2}
	\end{align}
\end{lemma}
\begin{proof}
Choose $\eta=\eta(\eps)>0$ small enough.
For an integer $\ell>0$ consider the event
	$$\cE_\ell=\cbc{\sum_{x,y\in V_n}\TV{\mu_{\G\marg\{x,y\}}-p_{\ell,\partial^\ell[\G,x]}\tensor p_{\ell,\partial^\ell[\G,y]}}<\eta^2 n^2}$$
If $\cM$ is $p$-symmetric, then $\lim_{n\to\infty}\pr\brk{\G\in\cE_\ell}$ for sufficiently large $\ell$.
Similarly, if the planted distribution is $p$-symmetric, then 
$\lim_{n\to\infty}\pr\brk{\hat\G_\ell\in\cE_\ell}$ for large $\ell$.

Hence, assume that $G\in\cE_\ell$.
Then by the triangle inequality, for any $\omega\in\Omega$,
	\begin{align*}
	\frac1n\sum_{x\in V_n}\abs{p_{\ell,\partial^\ell[G,x]}(\omega)-\mu_{G\marg x}(\omega)}
		&=\frac1{n^2}\sum_{x\in V_n}\abs{\brk{\sum_{y\in V_n}\sum_{\omega'\in\Omega}p_{\ell,\partial^\ell[G,x]}(\omega)p_{\ell,\partial^\ell[G,y]}(\omega')}-
			\brk{\sum_{y\in V_n}\sum_{\omega'\in\Omega}\mu_{G\marg x,y}(\omega,\omega')}}\leq\eta^2.
	\end{align*}
Therefore,
	\begin{align}\label{eqLemma_psymmetric2}
	\frac1n\sum_{x\in V_n}\TV{p_{\ell,\partial^\ell[G,x]}-\mu_{G\marg x}}&\leq\eta^2|\Omega|<\eta.
	\end{align}
Furthermore, by (\ref{eqLemma_psymmetric2}) and the triangle inequality,
	\begin{align}\label{eqLemma_psymmetric3}
	\frac1{n^2}\sum_{x,y\in V_n}\TV{\mu_{G\marg x}\tensor\mu_{G\marg y}-p_{\ell,\partial^\ell[G,x]}\tensor p_{\ell,\partial^\ell[G,y]}}\leq2 \eta.
	\end{align}
Since $G\in\cE_\ell$, (\ref{eqLemma_psymmetric3}) entails that
	\begin{align*}
	\frac1{n^2}\sum_{x,y\in V_n}\TV{\mu_{G\marg x}\tensor\mu_{G\marg y}-\mu_{G\marg\{x,y\}}}\leq 3 \eta<\eps,
	\end{align*}
i.e., $G$ is $(\eps,2)$-symmetric.
\end{proof}

\begin{lemma}\label{Lemma_judicious}
There is a number $\eps_0=\eps_0(\Delta,\Omega,\Psi,\Theta)$ such that for all
$0<\eps<\eps_0$, $\ell>0$ there exists $\chi>0$ such that for large enough $n$ the following is true.
If $G\in\cG(\cM_n)$ is a $(2\ell+5)$-acyclic factor graph such that
	\begin{align}\label{eqLemma_judicious1}
	\sum_{x\in V_n}\TV{\mu_{G\marg x}-p_{\ell,\partial^\ell[G,x]}}&<\eps^3n
	\end{align}
and $\mu_G$ is $(\chi,2)$-symmetric,
then $\bck{\vecone\{(G,\SIGMA)\mbox{ is $(\eps,\ell)$-judicious w.r.t.\ $p$}\}}_{G}\geq1/2$.
\end{lemma}
\begin{proof}
Pick $\delta=\delta(\ell,\eps)>0$ small, $\beta=\beta(\delta)$ and $\gamma = \gamma(\beta)$ smaller and
$\chi = \chi(\gamma)>0$ smaller still and assume that $n>n_0(\chi)$.
Let $\vec V_0$ be the partition of $V_n$ such that $x,y\in V_n$ belong to the same class iff $\partial^{\ell+2}[G,x]=\partial^{\ell+2}[G,y]$.
By \Thm~\ref{Thm_decomp} there exists a refinement $\vec V$ of $\vec V_0$ such that $\mu_G$ is $\gamma$-homogeneous
with respect to $(\vV,\vS)$ for some partition $\vS$ of $\Omega^n$  such that $\#\vV+\#\vS\leq N=N(\gamma)$.
We may index the classes of $\vec V$ as $V_{T,i}$ with $T=\partial^{\ell+2}[G,x]$ for all $x$ in the class and $i\in[N_T]$ for some integer $N_T$.

Let $J$ be the set of all $j\in[\#\vS]$ such that $\mu(S_j)\geq\delta^7/N$ and $\mu[\nix|S_j]$ is $\gamma$-regular.
Then by {\bf HM1}
	\begin{equation}\label{eqNothingMuch}
	\sum_{j\in J}\mu(S_j)\geq1-\delta^6.
	\end{equation}
Further, \Lem~\ref{Lemma_regularSymmetric} shows that $S_j$ is a $(\beta,2)$-state if $j\in J$.
Therefore, choosing $\chi$ small enough, we obtain from \Cor~\ref{Cor_states2} that
	\begin{align*}
	\frac1n\sum_{x\in V_n}\TV{\mu_{G\marg x}[\nix|S_j]-\mu_{G\marg x}}&<\delta^7
		\quad\mbox{
		for all $j\in J$.} 
	\end{align*}
Therefore, by (\ref{eqLemma_judicious1}) and the triangle inequality, for $j\in J$ 
we get
	\begin{align*}
	\frac1n\sum_{x\in V_n}\TV{\mu_{G\marg x}[\nix|S_j]-p_{\ell,\partial^\ell[G,x]}}\leq
	\eps^3+\frac1n\sum_{x\in V_n}\TV{\mu_{G\marg x}[\nix|S_j]-\mu_{G\marg x}}&<\eps^3+3\delta^7<2\eps^3.
	\end{align*}
Consequently, by (\ref{eqNothingMuch}), 
Bayes' rule and the triangle inequality,
summing over all $T\in\fT_{\ell+2}\cap\cV$ and $i\in[N_T]$ we get
	\begin{align}
	\frac1n \sum_{T,i}|V_{T,i}|\bck{\TV{\SIGMA[\nix|V_{T,i}]-p_{\ell,T}}}_G&=
		\frac1n \sum_{T,i}\sum_{j\in[\#\vS]}|V_{T,i}|\mu_G(S_j)\bck{\TV{\SIGMA[\nix|V_{T,i}]-p_{\ell,T}}|S_j}_G\nonumber\\
		&\leq\delta^7+\frac1n\sum_{T,i}\sum_{j\in[\#\vS]}|V_{T,i}|\mu_G(S_j)\TV{\bck{\SIGMA[\nix|V_{T,i}]|S_j}_G-p_{\ell,T}}
			\qquad\mbox{[by {\bf HM2}]}\nonumber\\
		&\leq\delta^7+\frac1n\sum_{T,i}\sum_{x\in V_{T,i}}\sum_{j\in[\#\vS]}
			\mu_G(S_j)\TV{\mu_{G\marg x}[\nix|S_j]-p_{\ell,\partial^\ell[G\marg x]}}<3\eps^3.
		\label{eqLemma_judicious2}
	\end{align}
Applying the triangle inequality once more, we find
	\begin{align}\label{eqLemma_judicious3}
	\sum_{T\in\fT_\ell\cap\cV}\lambda_{G,\ell}[T|\cV]\bck{\TV{q_{G,\SIGMA,\ell,T}-p_{\ell,T}}}&\leq
		\frac1n\sum_{T,i}|V_{T,i}|\bck{\TV{\SIGMA[\nix|V_{T,i}]-p_{\ell,T}}}_G<3\eps^3.
	\end{align}

Further, consider $T\in\fT_\ell\cap\cF$ such that $\lambda_{G,\ell}[T|\cF]>0$ and let $j\in[d_F]$.
Because $G$ is $(2\ell+5)$-acyclic, there exists a set $\Gamma(T,j)\subset\fT_{\ell+2}\cap\cV$ with the following two properties.
First, for every constraint node $a$ with $\partial^{\ell+1}[G,a]=T$ the variable node $x=\partial(G,a,j)$ satisfies $\partial^{\ell+2}[G,x]\in\Gamma(T,j)$.
Second, for every variable node $x$ with $\partial^{\ell+2}[G,x]\in\Gamma(T,j)$ there is a constraint node $a$
with $\partial^{\ell+1}[G,a]=T$ such that $\partial(G,a,j)=x$.
For $R\in\Gamma(T,j)$ let $m_{R,T,j}$ be the number of constraint nodes $a$ with $\partial^{\ell+1}[G,a]=T$ such that
$x=\partial(G,a,j)$ satisfies $\partial^{\ell+2}[G,x]=R$.
Then by the triangle inequality,
	\begin{align}
	&\hspace{-2cm}\sum_{T\in\fT_\ell\cap\cF}\sum_{j\in[d_T]}\lambda_{G,\ell}[T|\cF]\bck{\TV{q_{G,\SIGMA,\ell,T\marg j}-p_{\ell,\partial^\ell[T\reroot j]}}}_G
			\nonumber\\
		&\leq\sum_{T\in\fT_\ell\cap\cF}\sum_{j\in[d_T]}\sum_{R\in\Gamma(T,j)}\frac{m_{R,T,j}}{|F_n|}
			\bck{\TV{\SIGMA[\nix|V_{R,i}]-p_{\ell,R}}}_G\nonumber\\
		&\leq\frac{\Delta^2}n\sum_{R\in\fT_{\ell+2}\cap\cV}\sum_{i\in[d_R]}\bck{\TV{\SIGMA[\nix|V_{R,i}]-p_{\ell,R}}}_G;
	\label{eqLemma_judicious4}
	\end{align}
the last inequality follows because all degrees are between one and $\Delta$.
Finally, the assertion follows from (\ref{eqLemma_judicious2}), (\ref{eqLemma_judicious3}) and (\ref{eqLemma_judicious4}).
\end{proof}

We proceed by proving the upper bound and the lower bound statement from \Thm~\ref{Thm_symUpperBound} separately.
Strictly speaking, the proof of the lower bound implies the upper bound as well.
But presenting the arguments separately makes them slightly easier to follow.

\begin{proof}[Proof of \Thm~\ref{Thm_symUpperBound}, upper bound]
For $\eps,l>0$ let
	$\cE(\eps,l)=\{\sum_{x\in V_n}\TV{\mu_{\G\marg x}-p_{l,\partial^l[\G,x]}}<\eps n\}.$
Additionally, 
let  $\cS(\chi)$ be the event that $\mu_{\G}$ is $(\chi,2)$-symmetric
and let $\cL(\eps,l)$ be the event that $\TV{\lambda_{\G,l}-\thet_l}<\eps$.
We assume that $\underline\cM$ is $p$-symmetric.

Given $\alpha>0$ choose a small enough $\eps>0$ and a large enough $\ell>0$ as promised by \Cor~\ref{Cor_fmCalc}.
By \Lem~\ref{Lemma_martingale} there is $\ell_*>\ell$ such that
	\begin{align}\label{eqmyMartingale1}
	\Erw\brk{\TV{p_{l,\partial^l\T}-p_{\T}}|\cV}<\eps^4\qquad\mbox{ for all $l\geq\ell_*$}.
	\end{align}
Let $\chi=\chi(\eps,\ell_*)$ be the number provided by \Lem~\ref{Lemma_judicious}. 
Then \Lem~\ref{Lemma_psymmetric} implies that 
	$\lim_{n\to\infty}\pr\brk{\G\in\cS(\chi)}=1.$
Similarly, \Lem~\ref{Lemma_psymmetric} implies that  for large enough $l$ we have
	$\lim_{n\to\infty}\pr\brk{\G\in\cE(\eps^4,l)}=1$
Hence, 
the local convergence assumption~(\ref{eqLocalWeakConvergence3}) implies that for all large enough $l$,
	\begin{align}
	\lim_{n\to\infty}\pr\brk{\G\in\cS(\chi)\cap\cL(\eps^4,l)\cap\cE(\eps^4,l)}&=1\label{eqmyMartingale4}.
	\end{align}
Further, we claim that $\cL(\eps^4,l)\cap\cE(\eps^4,l)\subset\cL(\eps^4,\ell_*)\cap\cE(\eps^3,\ell_*)$.
Indeed, if $l \geq \ell_*$, then $\cL(\eps^4,l)\subset\cL(\eps^4,\ell_*)$.
Moreover, if $G\in \cL(\eps^4,l)\cap\cE(\eps^4,l)$, then with $\vec x\in V_n$ chosen uniformly at random we find
	\begin{align*}
	\Erw\TV{\mu_{G\marg\vec x}-p_{\ell_*,\partial^{\ell_*}[G,\vec x]}}&\leq
		\Erw\TV{\mu_{G\marg\vec x}-p_{l,\partial^{l}[G,\vec x]}}+\Erw\TV{p_{\ell_*,\partial^{\ell_*}[G,\vec x]}-p_{l,\partial^{l}[G,\vec x]}}\\
		&\leq\eps^4+\sum_{T\in\fT_l\cap\cV}\lambda_{G,l}(T)\TV{p_{l,T}-p_{\ell_*,\partial^{\ell_*}T}}\\
		&\leq\eps^4+2\TV{\thet_l[\nix|\cV]-\lambda_{G,l}[\nix|\cV]}+\Erw\brk{\TV{p_{l,\partial^l\T}-p_{\ell_*,\partial^{\ell_*}\T}}|\cV}<\eps^3.
	\end{align*}
Consequently, combining (\ref{eqmyMartingale1}) and (\ref{eqmyMartingale4}), we find that the event
$\cB(\alpha)=\cS(\chi)\cap\cL(\eps^4,\ell_*)\cap\cE(\eps^3,\ell_*)$ satisifes
	\begin{align}
	\lim_{n\to\infty}\pr\brk{\G\in\cB(\alpha)
	}&=1.\label{eqmyMartingale6}
	\end{align}
Further, if
$G\in\cB(\alpha)\cap\cA_{100\ell_*}$, 
then
$Z(G)\leq2Z(G)\bck{\vecone\{(G,\SIGMA)\mbox{ is $(\eps,\ell_*)$-judicious w.r.t.\ $p$}\}}_{G}$  by 
\Lem~\ref{Lemma_judicious} and the choice of $\chi$.
Therefore,
	\begin{align}\label{eqImjudicious_untensorised}
	\Erw[\vecone\{\G\in\cB(\alpha)\cap\cA_{100\ell_*}\}Z(\G)]&\leq2
		\Erw\brk{\vecone\{\G\in\cL(\eps^4,\ell_*)\cap\cA_{100\ell_*}\}
			\bck{\vecone\{(\G,\SIGMA)\mbox{ is $(\eps,\ell_*)$-judicious w.r.t.\ $p$}\}}_{\G}Z(\G)}.
	\end{align}
Since $\ell_*>\ell$, for large enough $n$ \Cor~\ref{Cor_fmCalc} and (\ref{eqImjudicious_untensorised}) yield
	\begin{align}\label{eqProofThm_symUpperBound2}
	\Erw[\vecone\{\G\in\cB(\alpha)\cap\cA_{100\ell_*}\}Z(\G)]&\leq2\exp(n(\cB_\thet(p)+\alpha)).
	\end{align}
Further, 
combining (\ref{eqmyMartingale6}) and (\ref{eqProofThm_symUpperBound2}) and using Markov's inequality,
we conclude that
	$$\lim_{n\to\infty}\pr\brk{Z(\G)>\exp(n(\cB_\thet(p)+2\alpha))|\cB(\alpha)\cap\cA_{100\ell_*}}=0.$$
Therefore, (\ref{eqmyMartingale6}), the high girth assumption and \Prop~\ref{Lemma_conc} yield
	\begin{equation}\label{eqProofThm_symUpperBound99}
	\lim_{n\to\infty}\pr\brk{Z(\G)>\exp(n(\cB_\thet(p)+2\alpha))}=0.
	\end{equation}
Finally, since $|n^{-1}\ln Z(\G)|$ is bounded by some number $C=C(\Delta,\Omega,\Psi,\Theta)>0$ by the definition (\ref{eqZ}) of $Z$,
(\ref{eqProofThm_symUpperBound99}) implies that
	$\limsup_{n\to\infty}n^{-1}\Erw[\ln Z(\G)]\leq\cB_\thet(p)+3\alpha$.
Taking $\alpha\to0$ completes the proof.
\end{proof}

To establish the lower bound we introduce a construction reminiscent of those used in~\cite{DMSS,DSS1,Galanis,MWW,SS}.
Namely, starting from the sequence $\underline\cM$ of $(\Delta,\Omega,\Psi,\Theta)$-models, we define another
 sequence $\underline\cM^\tensor=(\cM^\tensor_n)_n$ of models as follows.
Let $\Omega^\tensor=\Omega\times\Omega$ and let us denote pairs $(\omega,\omega')\in\Omega^\tensor$ by $\omega\tensor\omega'$.
Further, for any $\psi:\Omega^h\to(0,\infty)$ we define a function
	$$\psi^\tensor:(\Omega^{\tensor})^{h}\to(0,\infty),\qquad
		({\omega_1}\tensor{\omega_1'},\ldots,{\omega_{h}}\tensor{\omega_{h}'})\mapsto
			\psi(\omega_1,\ldots,\omega_{h})\cdot\psi(\omega_1',\ldots,\omega_{h}').$$
Let $\Psi^\tensor=\{\psi^\tensor:\psi\in\Psi\}$.
Then the $(\Delta,\Omega,\Psi,\Theta)$-model $\cM_n=(V_n,F_n,d_n,t_n,(\psi_a)_{a\in F_n})$ gives rise to the $(\Delta,\Omega^\tensor,\Psi^\tensor,\Theta)$-model
	$\cM_n^\tensor=(V_n,F_n,d,t,(\psi_a^\tensor)_{a\in F_n})$.

Clearly, there is a canonical bijection $\cG(\cM)\to\cG(\cM^\tensor)$, $G\mapsto G^\tensor$.
Moreover, the construction ensures that the Gibbs measure $\mu_{G^\tensor}\in\cP(\Omega^{\tensor\,n})$ equals $\mu_G\tensor\mu_G$.
Explicitly, for all $\omega_1,\omega_1',\ldots,\omega_n,\omega_n'\in\Omega$,
	\begin{align}\label{eqTensorConstruction1}
	\mu_{G^\tensor}(\omega_1\tensor\omega_1',\ldots,\omega_n\tensor\omega_n')=\mu_G(\omega_1,\ldots,\omega_n)\mu_G(\omega_1',\ldots,\omega_n').
	\end{align}
In effect, we obtain
	\begin{align}\label{eqTensorConstruction2}
	Z(G^\tensor)&=Z(G)^2.
	\end{align}
Further, writing $\fG^\tensor,\fT^\tensor$ for the $(\Delta,\Omega^\tensor,\Psi^\tensor,\Theta)$-templates and the acyclic
$(\Delta,\Omega^\tensor,\Psi^\tensor,\Theta)$-templates,
we can lift the marginal assignment $p$ from $\fT$ to $\fT^\tensor$ by letting $p^\tensor_{T^\tensor}=p_T\tensor p_T$ for all $T$.
Additionally, let $\thet^\tensor\in\cP(\fT^\tensor)$ be the image of  $\thet$ under the map $T\in\fT\mapsto T^\tensor\in\fT^\tensor$ so that
	\begin{align}\label{eqBetheSquare}
	\cB_{\thet^\tensor}(p^\tensor)&=2\cB_{\thet}(p).
	\end{align}

\begin{proof}[Proof of \Thm~\ref{Thm_symUpperBound}, lower bound]
We assume that $\underline\cM$ is $p$-symmetric and that the same is true of the planted distribution.
For $\eps,l>0$ consider the event
	\begin{align}\label{eqEevent}
	\cE^\tensor(\eps,l)&=\cbc{\frac1{n}\sum_{x\in V_n}\TV{\mu_{\G^\tensor\marg x}-p^\tensor_{\partial^l[\G^\tensor,x]}}<\eps}.
	\end{align}
and let  $\cS^\tensor(\chi)$ be the event that $\mu_{\G^\tensor}$ is $(\chi,2)$-symmetric.
Moreover, as before we let $\cL(\eps,\ell)=\{\TV{\lambda_{\G,\ell}-\thet_\ell}<\eps\}$.
Basically, we are going to apply the same argument as in the proof of the upper bound to the random factor graph
$\G^\tensor$ and to $\hat\G_\ell$ for a large enough $\ell$.

Hence, let $\alpha>0$.
Then \Cor~\ref{Cor_fmCalc} applied to $\underline\cM^\tensor$
yields a small $\eps=\eps(\alpha)>0$ and a large $\ell=\ell(\alpha)>0$.
Moreover, \Cor~\ref{Cor_fmCalc_lower} provides a large $\ell'(\alpha)>0$.
Further, by \Lem~\ref{Lemma_martingale} and (\ref{eqTensorConstruction1}) there exists $\ell_*>\ell+\ell'$ such that
	\begin{align}\label{eqmyMartingale1Tensor}
	\Erw\brk{\TV{p_{\ell,\partial^\ell\T}-p_{\T}}|\cV}+
		\Erw\brk{\TV{p^\tensor_{\ell,\partial^\ell\T^\tensor}-p^\tensor_{\T^\tensor}}|\cV}<\eps^4\qquad\mbox{ for all $\ell\geq\ell_*$}.
	\end{align}
Applying \Lem~\ref{Lemma_judicious} to $\underline\cM^\tensor$, we obtain $\chi_*=\chi_*(\eps,\ell_*)>0$ and
\Prop~\ref{Prop_tensorise} and \Lem~\ref{Lemma_psymmetric} imply that 
	\begin{align}\label{eqmyMartingale2Tensor}
	\lim_{n\to\infty}\pr\brk{\G\in\cS^\tensor(\chi_*)}&=1.
	\end{align}
Further, \Lem~\ref{Lemma_psymmetric} shows that for $l$ we have
	\begin{align}\label{eqmyMartingale3Tensor}
	\lim_{n\to\infty}\pr\brk{\G\in\cE^\tensor(\eps^4,l)}&
		=1.
	\end{align}
In effect, just as before (\ref{eqmyMartingale2Tensor}), (\ref{eqmyMartingale3Tensor}) and~(\ref{eqLocalWeakConvergence3}) show that large $l$,
	\begin{align}
	\lim_{n\to\infty}\pr\brk{\G\in\cS^\tensor(\chi_*)\cap\cL(\eps^4,l)\cap\cE^\tensor(\eps^4,l)}&=1\label{eqmyMartingale4Tensor}.
	\end{align}
Like in the upper bound proof we have
	$\cL(\eps^4,l)\cap\cE^\tensor(\eps^4,l)\subset
		\cL(\eps^4,l)\cap\cE^\tensor(\eps^3,\ell_*)$.
Therefore, (\ref{eqmyMartingale1Tensor}) and (\ref{eqmyMartingale4Tensor})  
show that the event 
$\cB^\tensor(\alpha)=\cS^\tensor(\chi_*)\cap\cL(\eps^4,\ell_*)\cap\cE^\tensor(\eps^3,\ell_*)$ 
satisfies
	\begin{align}
	\lim_{n\to\infty}\pr\brk{\G\in\cB^\tensor(\alpha)}&=1.\label{eqmyMartingale6Tensor}
	\end{align}

Define 
	$\cZ_{\alpha}(\G)=\vecone\{\G\in\cB^\tensor(\alpha)\cap\cA_{100\ell_*}\}Z(\G)$.
If $G\in\cB^\tensor(\alpha)\cap\cA_{100\ell_*}$, 
then by (\ref{eqTensorConstruction2}),
\Lem~\ref{Lemma_judicious} and the choice of $\chi_*$ we have
	\begin{align*}
	Z(G)^2&=Z(G^\tensor)\leq2Z(G^\tensor)
			\bck{\vecone\{(G^\tensor,\SIGMA)\mbox{ is $(\eps,\ell_*)$-judicious w.r.t.\ $p^\tensor$}\}}_{G^\tensor}.
	\end{align*}
Hence, we obtain an upper bound on the second moment of $\cZ_\alpha$, namely
	\begin{align}\label{eqImjudicious}
	\Erw[\cZ_{\alpha}(\G)^2]&\leq2
		\Erw\brk{\vecone\{\G\in\cL(\eps^4,\ell_*)\cap\cA_{100\ell_*}\}
			\bck{\vecone\{(\G^\tensor,\SIGMA)\mbox{ is $(\eps,\ell_*)$-judicious w.r.t.\ $p^\tensor$}\}}_{\G^\tensor}
			Z(\G^\tensor)}.
	\end{align}
Due to (\ref{eqBetheSquare}) and the choice of $\eps,\ell$ and because $\ell_*>\ell$,
\Cor~\ref{Cor_fmCalc} enables us to estimate the r.h.s.\ of (\ref{eqImjudicious}) explicitly, whence
	\begin{align}\label{eqProofThm_symUpperBound2}
	\Erw[\cZ_{\eps}(\G)^2]&\leq\exp(n(2\cB_\thet(p)+\alpha)).
	\end{align}

As a next step, we are going to 
show that
	\begin{align}\label{eqProofThm_symLowerBound1}
	\Erw[\cZ_{\eps}(\G)]\geq	\exp(n(\cB_\thet(p^\tensor)-2\alpha)).
	\end{align}
Indeed, by \Prop~\ref{Prop_tensorise}  and \Lem~\ref{Lemma_psymmetric} we have
	\begin{align}\label{eqmyMartingale2_lower}
	\lim_{n\to\infty}\pr\brk{\hat\G_l\in\cS^\tensor(\chi_*)}&=1
	\end{align}
 for large enough $l$.
Similarly, (\ref{eqTensorConstruction1}), the assumption that the planted distribution is $p$-symmetric and \Lem~\ref{Lemma_psymmetric} imply that  for $l$ large enough
	\begin{align}\label{eqmyMartingale3_lower}
	\lim_{n\to\infty}\pr\brk{\hat\G_l\in\cE^\tensor(\eps^4,l)}&=1.
	\end{align}
Hence, (\ref{eqmyMartingale2_lower}), (\ref{eqmyMartingale3_lower}), the local convergence assumption~(\ref{eqLocalWeakConvergence3})
and the construction (\ref{eqPlantedDistribution}) of the planted distribution imply that for $l$ large enough
	\begin{align}
	\lim_{n\to\infty}\pr\brk{\hat\G_l\in\cS^\tensor(\chi_*)\cap\cL(\eps^4,l)\cap\cE^\tensor(\eps^4,l)}&=1.\label{eqmyMartingale5_lower}
	\end{align}
Combining (\ref{eqmyMartingale1Tensor}) and (\ref{eqmyMartingale5_lower}) and using the high girth assumption, we thus obtain for large $l$ 
	\begin{align}
	\lim_{n\to\infty}\pr\brk{\hat\G_l\in\cB^\tensor(\alpha)
		}&=1.\label{eqmyMartingale7}
	\end{align}
Further, \Cor~\ref{Cor_fmCalc_lower} shows that
	\begin{align*}
	\lim_{n\to\infty}\pr\brk{
		\frac1n\ln\Erw\brk{Z(\G)|\cT_{l}}\geq\cB_\thet(p)-\alpha\bigg|\cA_{100l}}=1.
	\end{align*}
Thus, (\ref{eqmyMartingale7}) and \Prop~\ref{Prop_plantedModel} yield (\ref{eqProofThm_symLowerBound1}).

Finally, combining (\ref{eqProofThm_symUpperBound2}) and (\ref{eqProofThm_symLowerBound1}) and applying the Paley-Zygmund inequality, we obtain
	\begin{align}\label{eqProofThm_symLowerBound2}
	\pr\brk{Z(\G)\geq\exp(n(\cB_\thet(p)-4\alpha))}&\geq\pr\brk{\cZ_\eps(\G)\geq\exp(n(\cB_\thet(p)-4\alpha))}
		\geq\frac{\Erw[\cZ_\eps(\G)]^2}{2\Erw[\cZ_\eps(\G)^2]}\geq\exp(-10\alpha n).
	\end{align}
As this holds for any $\alpha>0$, the assertion follows from (\ref{eqProofThm_symLowerBound2}) and \Prop~\ref{Lemma_conc}.
\end{proof}

\subsection{Proof of \Thm~\ref{Thm_nonReconstruction}}\label{Sec_Thm_nonReconstruction}
The key step of the proof is to establish the following statement.

\begin{lemma}\label{Lemma_nonRe}
For any $\eps>0$ there exists $\delta>0$ such that for any $\ell>0$ there exists $n_0$ such that for all $n>n_0$ the following is true.
Assume that $G\in\cG(\cM_n)$ satisfies
	\begin{align}\label{eqThm_nonReconstruction1}
	\frac1n\sum_{x\in V_n}\bck{\TV{\bck{\SIGMA[\nix|x]|\nabla_\ell(G,x)}_\mu-p_{\ell,\partial^\ell[G,x]}}}_\mu<\delta^9.
	\end{align}
Then $G$ is $(\eps,2)$-symmetric and
	$\sum_{x\in V_n}\TV{\mu_{G\marg x}-p_{\ell,\partial^\ell[G,x]}}<\eps n$.
\end{lemma}

\noindent
Before we prove \Lem~\ref{Lemma_nonRe} let us show how it implies \Thm~\ref{Thm_nonReconstruction}.

\begin{proof}[Proof of \Thm~\ref{Thm_nonReconstruction}]
If $G\in\cG(\cM_n)$ satisfies is $(\eps,2)$-symmetric and $\sum_{x\in V_n}\TV{\mu_{G\marg x}-p_{\ell,\partial^\ell[G,x]}}<\eps n$, then by the triangle inequality
	\begin{align*}
	\sum_{x,y\in V_n}\TV{\mu_{G\marg\{x,y\}}-p_{\ell,\partial^\ell[G,x]}\tensor p_{\ell,\partial^\ell[G,y]}}&
		\leq\sum_{x,y\in V_n}\TV{\mu_{G\marg\{x,y\}}-\mu_{G\marg x}\tensor\mu_{G\marg y}}+
		\TV{\mu_{G\marg x}\tensor\mu_{G\marg y}-p_{\ell,\partial^\ell[G,x]}\tensor p_{\ell,\partial^\ell[G,y]}}\\
		&\leq4\eps n^2.
	\end{align*}
Therefore, the theorem follows by applying \Lem~\ref{Lemma_nonRe} either to the random factor graph $\G$ or
to the random factor graph $\G_\ell$ chosen from the planted model.
\end{proof}

\begin{proof}[Proof of \Lem~\ref{Lemma_nonRe}]
Let $\gamma=\gamma(\eps)>0$ be sufficiently small.
By \Thm~\ref{Thm_decomp} we can pick $\delta=\delta(\gamma)>0$ small enough so that there exists a partition
$(\vec V,\vec S)$ with $\#\vec V+\#\vec S<\delta^{-1}$ with respect to which $\mu_G$ is $\gamma^4$-homogeneous.
Suppose that $V_i$, $S_j$ are classes such that $|V_i|\geq\delta^{3/2}n$, $\mu_G(S_j)\geq\delta^{3/2}$
and such that $\mu[\nix|S_j]$ is $\gamma^4$-regular on $V_i$.
We claim that
	\begin{align}\label{eqMyContradiction}
	\frac1{|V_i|}\sum_{x\in V_i}\TV{\mu_{G\marg x}[\nix|S_j]-p_{\ell,\partial^\ell[G,x]}}&<3\gamma.
	\end{align}
The assertion is immediate from this inequality.
Indeed, suppose that (\ref{eqMyContradiction}) is true for all $i,j$ such that $|V_i|\geq\delta^{3/2}n$, $\mu_G(S_j)\geq\delta^{3/2}$
such that $\mu[\nix|S_j]$ is $\gamma^4$-regular on $V_i$.
Then 
 because $\#\vV+\#\vS\leq1/\delta$
	\begin{align}\label{eqMyContradiction666}
	\sum_{x\in V_n}\TV{\mu_{G\marg x}[\nix|S_j]-p_{\ell,\partial^\ell[G,x]}}<4\gamma n.
	\end{align}
Hence, by {\bf HM1} and Bayes' rule,
$\sum_{x\in V_n}\TV{\mu_{G\marg x}-p_{\ell,\partial^\ell[G,x]}}<5\gamma n<\eps n$.
Further, (\ref{eqMyContradiction666}) and \Lem~\ref{Lemma_regularSymmetric} imply that $\mu_G$ is $(\eps,2)$-regular (provided that
we pick $\gamma$ small enough).
Thus, we are left to prove (\ref{eqMyContradiction}).

Assume for contradiction that (\ref{eqMyContradiction}) is violated
for $V_i$, $S_j$ such that $|V_i|\geq\delta^{3/2}n$, $\mu_G(S_j)\geq\delta^{3/2}$.
Then  by the triangle inequality there is a set $W\subset V_i$ of size at least $\gamma|V_i|$ such that for all $x\in W$ we have
	\begin{align*}
	\TV{\mu_{G\marg x}[\nix|S_j]-p_{\ell,\partial^\ell[G,x]}}&\geq\gamma.
	\end{align*}
For $x\in W$ pick $\omega_x\in\Omega$ such that
	$|\mu_{G\marg x}[\omega_x|S_j]-p_{\ell,\partial^\ell[G,x]}|\geq\gamma$ is maximum.
Then by the pigeonhole principle there exist $\omega\in\Omega$ and $W'\subset W$, $|W'|\geq|W|/(2|\Omega|)$, such that either
	\begin{align}\label{eqLeadAstray1_prime}
	\forall x\in W':\mu_{G\marg x}[\omega|S_j]\geq p_{\ell,\partial^\ell[G,x]}(\omega)+\gamma
		&\qquad\mbox{or}\\
	\forall x\in W':\mu_{G\marg x}[\omega|S_j]\leq p_{\ell,\partial^\ell[G,x]}(\omega)-\gamma
			\label{eqLeadAstray2}
	\end{align}
In particular we have
	\begin{align}\label{eqLeadAstray1}
	\forall x\in W':\mu_{G\marg x}[\omega|S_j]\geq p_{\ell,\partial^\ell[G,x]}(\omega)+\gamma/|\Omega|
	\end{align}

We claim that there is a set $L\subset W'$ of size $|L|=\lceil1/\delta\rceil$ with the following properties.
	\begin{enumerate}[(i)]
	\item the pairwise distance between any two $x,y\in L$ is at least $10(\ell+1)$.
	\item for all $x\in L$ we have
		\begin{align}\label{eqW''}
		\bck{\TV{\bck{\SIGMA[\nix|x]|\nabla_\ell(G,x)}_{G}-p_{\ell,\partial^\ell[G,x]}}}_{\mu_G}<\delta^4.
		\end{align}
	\end{enumerate}
Indeed,
because $|V_i|\geq\delta^2n$ and $\mu(S_j)\geq\delta^2$ the assumption (\ref{eqThm_nonReconstruction1}) implies that
	\begin{align}\label{eqThm_nonReconstruction2}
	\sum_{x\in V_i}\bck{\TV{\bck{\SIGMA[\nix|x]|\nabla_\ell(G,x)}_{G}-p_{\ell,\partial^\ell[G,x]}}}_{\mu_G[\nix|S_j]}<\delta^5|V_i|.
	\end{align}
Since $|W'|\geq\gamma|V_i|/|\Omega|\geq\delta|V_i|$,
 (\ref{eqThm_nonReconstruction2}) implies that there is a set $W''\subset W'$ of size $|W''|\geq|W'|/2$ such that (\ref{eqW''}) holds for all $x\in W''$.
Now, construct a sequence $W''=W_0''\supset W_1''\cdots$ inductively as follows.
In step $i\geq1$ pick some $x_i\in W_{i-1}''$.
Then $W_i''$ contains $x_i$ and all $y\in W_{i-1}''\setminus\{x_i\}$ whose distance from $x_i$ is greater than $10(\ell+1)$.
Since for each $x_i$ the total number of variable nodes at distance at most $10(\ell+1)$ is bounded by $\Delta^{10(\ell+1)}$ and
	$|W_0''|\geq\delta|V_i|/2\geq\delta^3n/2$, the set $\bigcap_{i\geq1}W_i''$ has size at least $\delta^3\Delta^{-10(\ell+1)}n/2>1/\delta$,
provided that $n$ is large enough.
Finally, simply pick any subset $L\subset\bigcap_{i\geq1}W_i''$ of size $|L|=\lceil1/\delta\rceil$.

Consider the event $\cE=\cbc{\SIGMA[\omega|L]\geq|L|^{-1}\sum_{x\in L}p_{\ell,\partial^\ell[G,x]}+\gamma^3}.$
We claim that
		\begin{align}\label{eqfrakL4}
		\mu_G[\cE|S_j]\leq2\delta^2.
		\end{align}
Indeed, by (\ref{eqW''}) and the union bound we have
	\begin{align}\nonumber
		\bck{\vecone\cbc{\forall x\in L:\TV{\bck{\SIGMA[\nix|x]|\nabla_\ell(G,x)}_{G}-p_{\ell,\partial^\ell[G,x]}}\leq\delta}}_{\mu_G}&\geq
			1-\sum_{x\in L}\bck{\vecone\cbc{\TV{\bck{\SIGMA[\nix|x]|\nabla_\ell(G,x)}_{G}-p_{\ell,\partial^\ell[G,x]}}>\delta}}_{\mu_G}\\
			&\geq1-\delta^2.
				\label{eqfrakL1}
	\end{align}
Now, let $\mathfrak L$ be the coarsest $\sigma$-algebra such that $\mathfrak L\supset\nabla_\ell(G,x)$ for all $x\in L$.
Suppose that $\sigma\in S_j$ is such that
	\begin{align}\label{eqfrakL}
	\TV{\bck{\SIGMA[\nix|x]|\nabla_\ell(G,x)}_{G}(\sigma)-p_{\ell,\partial^\ell[G,x]}}\leq\delta\quad\mbox{ for all $x\in L$}.
	\end{align}
We claim that (\ref{eqfrakL}) implies
	\begin{align}\label{eqfrakL2}
	\bck{\vecone\{\SIGMA\in\cE\}|\mathfrak L}_{G}(\sigma)&<\delta^3.
	\end{align}
Indeed, let $X=\sum_{x\in L}\vecone\{\SIGMA(x)=\omega\}$.
Then (\ref{eqfrakL}) implies that 
		\begin{align}\label{eqfrakL3}
		\bck{X(\SIGMA)|\mathfrak L}(\sigma)\leq2\delta|L|+\sum_{x\in L}p_{\ell,\partial^\ell[G,x]}(\omega).
		\end{align}
Furthermore, the pairwise distance of the variables in $L$ is at least $2(\ell+1)$ and given $\mathfrak L$
the values of the variables at distance either $\ell$ or $\ell+1$ from each $x\in L$ are fixed.
Therefore, given  $\mathfrak L$ the events $\{\SIGMA(x)=\omega\}$ are mutually independent.
In effect, $X$ is stochastically dominated by a sum of independent random variables.
Hence, recalling that $\delta$ is much smaller than $\gamma$, we see that (\ref{eqfrakL2}) follows from (\ref{eqfrakL3}) and the Chernoff bound.
Finally, combining (\ref{eqfrakL1}) and (\ref{eqfrakL2}) we obtain (\ref{eqfrakL4}).

But (\ref{eqfrakL4}) does not sit well with (\ref{eqLeadAstray1}).
In fact, (\ref{eqLeadAstray1}) entails that  
$\mu_G[\cE|S_j]\geq\gamma^2$;
for consider the random variable $Y=\sum_{x\in L}\vecone\{\SIGMA(x)\neq\omega\}$.
Then (\ref{eqLeadAstray1}) yields 
	$\bck{Y}_{\mu[\nix|S_j]}\leq\sum_{x\in L}(1-\mu_{G \marg x }[\omega|S_j]) \leq|L|(1-\gamma/|\Omega|)-\sum_{x\in L}p_{\ell,\partial^\ell[G,x]}(\omega)$.
Hence, by Markov's inequality
	$$1-\mu_G[\cE|S_j]\leq\frac{\bck{Y}_{\mu[\nix|S_j]}}{|L|(1-\gamma^3)-\sum_{x\in L}p_{\ell,\partial^\ell[G,x]}(\omega)}
		\leq\frac{|L|(1-\gamma/|\Omega|)-\sum_{x\in L}p_{\ell,\partial^\ell[G,x]}(\omega)}{|L|(1-\gamma^3)-\sum_{x\in L}p_{\ell,\partial^\ell[G,x]}(\omega)}
			\leq\frac{1-\gamma/|\Omega|}{1-\gamma^3}\leq1-\gamma^2.$$
Combining this bound with (\ref{eqfrakL4}), we obtain
	$\gamma^2\leq{\mu_G(\cE)}/{\mu_G(S_j)}\leq2\delta^2/\mu_G(S_j)$.
Thus, choosing $\delta$ much smaller than $\gamma$, we conclude that $\mu_G(S_j)<\delta^{3/2}$, which is a contradiction.
Thus, we have established that (\ref{eqMyContradiction}).
\end{proof}

\section{Conditioning on the local structure}\label{Sec_crazyConfigs}

\subsection{A generalised configuration model}
The aim in this section is to prove \Prop~\ref{Lemma_fmCalc}.
The obvious problem is the conditioning on the $\sigma$-algebra $\cT_\ell$ that fixes the depth-$\ell$ neighborhoods
of all variable nodes and the depth-$\ell+1$ neighborhoods of all constraint nodes.
Following~\cite{BordenaveCaputo}, we deal with this  conditioning by setting up a generalised configuration model.

Recall that $\fT_\ell$ is the (finite) set of all isomorphism classes $\partial^\ell T$ for $T\in\fT\cap\cV$ and $\partial^{\ell+1} T$ for $T\in\fT\cap\cV$.
Let $\ell,n>0$ be integers and let $\cM=(V,F,d,t,(\psi_a)_{a\in F})$ be a $(\Delta,\Omega,\Psi,\Theta)$-model of size $n$.
Moreover, let $G\in\cG(\cM)$ be a $100\ell$-acyclic factor graph.
Then we define an enhanced $(\Delta,\Omega,\Psi,\Theta_\ell)$-model  $\cM(G,\ell)$ with type set $\Theta_\ell=(\fT_\ell\cap\cV)\times[\Delta]$ as follows.
The set of variable nodes is $V$, the set of constraint nodes is $F$, the degrees are given by $d$ and the weight function
associated with each constraint $a$ is $\psi_a$ just as in $\cM$.
Moreover, the type of a variable clone $(x,i)$ is 
	$t_{G,\ell}(x,i)=(\partial^\ell[G,x],i)$.
Further, the type of a constraint clone $(a,j)$ such that $\partial(G,a,j)=(x,i)$ is $t_{G,\ell}(a,j)=(\partial^\ell[G,x],i)$.
Clearly, $\cG(\cM(G,\ell))\subset\cG(\cM)$.
The following lemma shows that the model $\cM(G,\ell)$ can be used to generate factor graphs whose local structure coincides with that of $G$.

\begin{lemma}\label{Lemma_beMyType}
Assume that $\ell\geq0$ and that $G'\in\cG(\cM(G,\ell))$ is $2\ell+4$-acyclic.
Then $G'$ viewed as a $\cM$-factor graph satisfies $G\ism_\ell G'$.
\end{lemma}
\begin{proof}
We are going to show inductively for $l\in[\ell]$ that $G\ism_l G'$.
The case $l=0$ is immediate from the construction.
Thus, assume that $l>0$, let $(x,i)\in C_V$ and let $B$ be the set of all clones that have distance precisely $l-1$ from $(x,i)$.
Since $G'$ is $(2\ell+2)$-acyclic, the pairwise distance of any two clones in $B$ is at least $2$.
Moreover, by induction we know that $t_{G,1}(w,j)=t_{G',1}(w,j)$ for all $(w,j)\in B$.
Therefore, $t_{G,l}(x,i)=t_{G',l}(x,i)$.
\end{proof}

In order to prove \Prop~\ref{Lemma_fmCalc} we need to enhance the model $\cM(G,\ell)$ further to accommodate an assignment
that provides a value from $\Omega$ for each clone.
Thus, let $\hat\sigma:C_V\cup C_F\to\Omega$ be a map.
We call $\hat\sigma$ {\bem valid} if $\hat\sigma(x,i)=\hat\sigma(x,j)$ for all $x\in V$, $i,j\in[d(x)]$
and if for all $\theta\in\Theta_\ell$ we have
	$$\forall\omega\in\Omega:
		\abs{\cbc{(x,i)\in C_V:\hat\sigma(x,i)=\omega,t_{G,\ell}(x,i)=\theta}}
		=\abs{\cbc{(a,j)\in C_F:\hat\sigma(a,j)=\omega,t_{G,\ell}(a,j)=\theta}}.$$
Of course, we can extend a valid $\hat\sigma$ to a map $V\to\Omega$, $x\mapsto\hat\sigma(x,1)$.
Given a valid $\hat\sigma$
we define a model $(\Delta,\Omega,\Psi,\Theta_\ell\times\Omega)$-model $\cM(G,\hat\sigma,\ell)$ with
variable nodes $V$, constraint nodes $F$, degrees $d$ and weight functions $(\psi_a)_{a\in F}$ such that the type $t_{G,\hat\sigma,\ell}(x,i)$
of a variable clone $(x,i)$ is $(\partial^\ell[G,x],i,\hat\sigma(x,i))$ and such that the type
$t_{G,\hat\sigma,\ell}(a,j)$ of a constraint clone $(a,j)$ with $\partial(G,a,j)=(x,i)$
is $(\partial^\ell[G,x],i,\hat\sigma(a,j))$.
By construction, 
 $\cG(\cM(G,\hat\sigma,\ell))\subset\cG(\cM(G,\ell))\subset\cG(\cM)$.
Let us recall the definition of the distance from (\ref{eqDist}).
Further, for two maps $\hat\sigma,\hat\sigma':C_V\cup C_F\to\Omega$
let $\dist(\hat\sigma,\hat\sigma')=|\cbc{(v,i)\in C_V\cup C_F:\hat\sigma(v,i)\neq\hat\sigma'(v,i))}|$.
In \Sec~\ref{Sec_getBack} we are going to establish the following.

\begin{lemma}\label{Lemma_getBack}
For any $\eps,\ell>0$ there is $n_0=n_0(\eps,\ell,\Delta,\Omega,\Psi,\Theta)$ such that for $n>n_0$ the following holds.
If $\cM$ is a $(\Delta,\Omega,\Psi,\Theta)$-model of size $n$, $G\in\cG(\cM)$ is $100\ell$-acyclic
and $\hat\sigma$ is valid, then with probability at least $1-\eps$ the random factor graph $\G(\cM(G,\hat\sigma,\ell))$ has the following property.
There exist a valid $\hat\sigma'$ and a $4\ell$-acyclic $G'\in\cG(\cM(G,\hat\sigma',\ell))$ such that
	$\dist(\hat\sigma,\hat\sigma')+
		\dist(G',\G(\cM(G,\hat\sigma,\ell)))\leq n^{0.9}.$
\end{lemma}

To proceed consider a $(G,\ell)$-marginal sequence $q$.
We call $\hat\sigma$ {\bem $q$-valid} if the following two conditions hold.
\begin{description}
\item[V1] For all $T\in\fT_\ell\cap\cV,\omega\in\Omega$ we have
	\begin{align*}
	\abs{\cbc{x\in V:\partial^\ell[G,x]=T,\hat\sigma(x)=\omega}}&=q_T(\omega)\abs{\cbc{x\in V:\partial^\ell[G,x]=T}}.
	\end{align*}
\item[V2] For all $T\in\fT_\ell\cap\cF,\omega_1,\ldots,\omega_{d_F}\in\Omega$ we have
	\begin{align*}
	\abs{\cbc{a\in F:\partial^{\ell+1}[G,a]=T,\forall j\in[d_F]:\hat\sigma(a,j)=\omega_j}}&=q_T(\omega_1,\ldots,\omega_{d_T})
		\abs{\cbc{a\in F:\partial^{\ell+1}[G,a]=T}}.
	\end{align*}
\end{description}

\begin{lemma}\label{Lemma_countingConfigurations}
For any $\eps,\ell>0$ there is $n_0=n_0(\eps,\ell,\Delta,\Omega,\Psi,\Theta)$ such that for $n>n_0$ the following holds.
Assume that $\cM$ is a $(\Delta,\Omega,\Psi,\Theta)$-model of size $n$, $G\in\cG(\cM)$ is $100\ell$-acyclic and $q$ is a 
$(G,\ell)$-marginal sequence such that there exists a $q$-valid $\hat\sigma$.
Then with the sum ranging over all $q$-valid $\hat\sigma$ we have
	\begin{align*}
	\exp\bc{n\cB_G(\ell,q)-\sqrt n}\leq\sum_{\hat\sigma}\frac{|\cG(\cM(G,\hat\sigma,\ell))|}{|\cG(\cM(G,\ell))|}
		\leq\exp\bc{n\cB_G(\ell,q)+\sqrt n}.
	\end{align*}
\end{lemma}

\noindent
We defer the proof of \Lem~\ref{Lemma_countingConfigurations} to \Sec~\ref{Sec_countingConfigurations}.

\begin{proof}[Proof of \Prop~\ref{Lemma_fmCalc}]
We claim that
	\begin{equation}\label{eqProp_fmCalc1}
	\abs{\cbc{G'\in\cG(\cM(n)):G'\ism_\ell G}}\geq|\cG(\cM(G,\ell))|\exp(-n^{0.91}).
	\end{equation}
To see this, apply \Lem~\ref{Lemma_getBack} to the constant map $\hat\sigma:(v,j)\in C_V\cup C_F\mapsto\omega_0$ for some fixed $\omega_0\in\Omega$.
Then we conclude that with probability at least $1/2$ the random graph $\G(\cM(G,\ell))=\G(\cM(G,\hat\sigma,\ell))$ is at distance at most $n^{0.9}$ from a $4\ell$-acyclic
$G'\in\G(\cM(G,\ell))\subset\cG(\cM)$.
Furthermore, by \Lem~\ref{Lemma_beMyType} this factor graph $G'$, viewed as an element of $\cG(\cM)$, satisfies $G\ism_\ell G'$.
Finally, since the total number of factor graphs at distance at most $n^{0.9}$ from $G'$ is bounded by $\exp(n^{0.91})$ because all degrees are bounded,
we obtain (\ref{eqProp_fmCalc1}).

Let $\delta>0$ be small enough. 
If $\sigma\in\Sigma(G,\ell,q,\delta)$, 
then by (\ref{eqZellqdelta})  there exists a $(G,\ell)$-marginal sequence $q'$ such that $\sigma\in\Sigma(G,\ell,q',0)$
such that $\TV{q_T-q_T'}<\delta$ for all $T\in\fT_\ell$.
Because $\fT_\ell$ is finite and $\Sigma(G,\ell,q',0)\neq\emptyset$, the total number of such $q'$ is bounded by a polynomial in $n$.
Moreover, due to the continuity of $\cB_{G,\ell}(\nix)$ we can choose $\delta=\delta(\ell)$ small enough so that
$|\cB_{G,\ell}(q')-\cB_{G,\ell}(q)|<\eps/2$ for all such $q'$.
Hence, 
summing over all $\hat\sigma$ corresponding to $\sigma\in\Sigma(G,\ell,q,\delta)$, 
we obtain from (\ref{eqProp_fmCalc1}) and  \Lem~\ref{Lemma_countingConfigurations} that
	\begin{align*}
	\Erw[Z_{\ell,q}(\G)|\G\ism_\ell G]&\leq\sum_{\hat\sigma}\frac{|\cG(\cM(G,\hat\sigma,\ell))|}{\abs{\cbc{G'\in\cG(\cM(n)):G'\ism_\ell G}}}
		\leq\exp(n\cB_G(\ell,q)+\eps n).
	\end{align*}
Conversely, by \Lem~\ref{Lemma_getBack} with probability at least $1/2$ the graph $\G(\cM(G,\hat\sigma,\ell))$ is within distance
at most $n^{0.9}$ of a $4\ell$-acyclic $G'$, which satisfies $G'\ism_\ell G$ by \Lem~\ref{Lemma_beMyType}.
As before, the total number of graphs at distance at most $n^{0.9}$ off $G'$ is bounded by $\exp(n^{0.91})$.
Similarly, the total number of $\hat\sigma'$ at distance at most $n^{0.9}$ off $\hat\sigma$ is bounded by $\exp(n^{0.91})$.
Therefore, by \Lem~\ref{Lemma_beMyType}
	\begin{align*}
	\Erw[\vecone\{\cA_{2\ell+1}\}Z_{\ell,q}(\G)|\G\ism_\ell G]&\geq
		\frac{\exp(-2n^{0.98})}2\sum_{\hat\sigma}\frac{|\cG(\cM(G,\hat\sigma,\ell))|}
			{\abs{\cG(\cM(G,\ell))}}
		\geq\exp(n\cB_G(\ell,q)-\eps n),
	\end{align*}
as desired.
\end{proof}

\subsection{Proof of \Lem~\ref{Lemma_getBack}}\label{Sec_getBack}
Let $\Theta_*=\cbc{t_{G,\hat\sigma,\ell}(x,i):(x,i)\in C_V}$ be the set of all possible types.
For each $\tau\in \Theta_*$ let $n_\tau$ be the number of clones $(x,i)\in C_V$ with $t_{G,\hat\sigma,\ell}(x,i)=\tau$.
Throughout this section we assume that $n>n_0(\eps,\ell,\Delta,\Omega,\Psi,\Theta)$ is sufficiently large.

\begin{lemma}\label{Lemma_getBack1}
There exists $\beta>0$ such that  the following is true.
For any $G,\hat\sigma$ there exists $3/4<\gamma<7/8$ such that for every $\tau\in \Theta_*$ either 
$n_\tau\leq n^\gamma$ or $n_\tau>n^{\gamma+\beta}$.
\end{lemma}
\begin{proof}
The number of possible types is bounded independently of $n$.
Hence, choosing $\beta$ small enough, we can ensure that there exists an integer $j>0$ such that $3/4+j\beta<7/8$
such that $[n^{3/4+j\beta},n^{3/4+(j+1)\beta}]\cap\cbc{n_\tau:\tau\in T}=\emptyset$.
\end{proof}

Fix $\beta,\gamma$ as in the previous lemma.
Call $\tau$ {\bem rare} if $n_\tau\leq n^\gamma$ and {\bem common} otherwise.
Let $Y$ be the number of variable clones that belong to cycles of length at most $10\ell$ in $\G(\cM(G,\hat\sigma,\ell))$.

\begin{lemma}\label{Lemma_getBack2}
For large enough $n$ we have $\Erw[Y]\leq n^\gamma\ln n$.
\end{lemma}
\begin{proof}
Let $R$ be the set of variable clones $(v,i)$ of a rare type and
let $U$ be the set of all variable clones whose distance from $R$ in $G$ does not exceed $100\ell$.
Since the maximum degree as well as the total number of types are bounded, we have $|U|\leq|R|\ln\ln n\leq n^{\gamma}\sqrt{\ln n}$,
provided that $n$ is big enough.
Thus, to get the desired bound on $\Erw[Y]$ we merely need to consider the set $W$ of common clones that are at distance more than $100\ell$ from $R$.

More specifically, let $(v,i)$ be a common clone.
We are going to bound the probability that $(v,i)\in W$ and that $(v,i)$ lies on a cycle of length at most $10\ell$.
To this end, we are going to explore the (random) factor graph from $(v,i)$ via the principle of deferred decisions.
Let $i_1=i,\ldots,i_l\in[\Delta]$ be a sequence of $l\leq10\ell$ indices.
If $(v,i)$ lies on a cycle of length at most $10\ell$, then there exists such a sequence $(i_1,\ldots,i_l)$ that corresponds to this cycle.
Namely, with $v_1=v$ the cycle comprises of the clones $(v_1,i_1),\ldots,(v_{l},i_{l})$ such that $\partial(\G(\cM(G,\hat\sigma,\ell)),v_j,i_j)=(v_{j+1},i_{j+1})$.
In particular, $v_l=v_1$.
Clearly, the total number of sequences $(i_1,\ldots,i_l)$ is bounded.
Furthermore, given that $(v_l,i_l)$ is common, the probability that $v_l=v_0$ is bounded by $2n^{-\gamma}$.
Since $\gamma>3/4$, the linearity of expectation implies that $\Erw[Y]\leq|U|+2n^{1-\gamma}\ln n\leq n^\gamma\ln n$.
\end{proof}

\begin{lemma}\label{Lemma_getBack3}
Assume that $G''\in\cG(\cM(G,\hat\sigma,\ell))$ satisfies $Y(G'')\leq n^\gamma\ln^2 n$.
Then there is a $4\ell$-acyclic $G'\in\cG(\cM(G,\ell))$ such that $\dist(G',G'')\leq n^{0.9}$.
\end{lemma}
\begin{proof}
Let $R$ be the set of variable clones $(v,i)$ of a rare type and
let $U$ be the set of all variable clones whose distance from $R$ in $G$ does not exceed $10\ell$.
Moreover, let $G'''\in\cG(\cM(G,\ell))$ minimise $\dist(G'',G''')$ subject to the condition that $\partial(G''',v,i)=\partial(G,v,i)$ for all $(v,i)\in U$.
Then $\dist(G'',G''')\leq n^\gamma\ln n$ because the total number of types is bounded.
Therefore, the assumption $Y(G'')\leq n^\gamma\ln^2 n$ implies that $Y(G''')\leq n^\gamma\ln^3 n$, say.
In addition, because $G$ is $100\ell$-acyclic, none of the clones in $R$ lies on a cycle of length at most $4\ell$ in $G'''$.

Altering only a bounded number of edges in each step, we are now going to remove the short cycles of $G'''$ one by one.
Let $C$ be the set of common clones.
The construction of $G'''$ ensures that only common clones lie on cycles of length at most $4\ell$.
Consider one such clone $(v,i)$ and let $N$ be the set of all variable clones that can be reached from $(v,i)$ by traversing precisely two edges of $G'''$;
	thus, $N$ contains all clones $(w,j)$ such that $w$ has distance two from $v$ and all clones $(v,j)$ that are incident to the same constraint node as $(v,i)$.
Once more by the construction of $G'''$ we have $N\subset C$.
Furthermore, $|N|\leq\Delta^2$.

We claim that there exists $N'\subset C$ and a bijection $\xi:N\to N'$ such that the following conditions are satisfied.
	\begin{enumerate}[(i)]
	\item $t_{G,\hat\sigma,\ell}(w,j)=t_{G,\hat\sigma,\ell}(\xi(w,j))$ for all $(w,j)\in N$.
	\item the pairwise distance in $G'''$ between any two clones in $N'$ is at least $100\ell$.
	\item the distance in $G'''$ between $N\cup \{(v,i)\}$ and $N'$ is at least $100\ell$.
	\item the distance between $R$ and $N'$ is at least $100\ell$.
	\item any $(w,j)\in N'$ is at distance at least $100\ell$ from any clone that belongs to a cycle of $G'''$ of length at most $4\ell$.
	\end{enumerate}
Since the maximum degree of $G'''$ is bounded by $\Delta$, there are no more than $n^\gamma\ln^4 n$ clones violate condition (iii), (iv) or (v).
By comparison, there are at least $n^{\gamma+\beta}$ clones of any common type.
Hence, the existence of $\xi$ follows.

Now, obtain $G''''$ from $G'''$ as follows.
\begin{itemize}
\item let $G''''(\xi(w,j))=G'''(w,j)$ and $G''''(w,j)=G'''(\xi(w,j))$ for all $(w,j)\in N$.
\item let $G''''(w,j)=G'''(w,j)$ for all $(w,j)\not\in N\cup N'$.
\end{itemize}
It is immediate from the construction that any clone on a cycle of length at most $4\ell$ in $G''''$ also lies on such a cycle of $G'''$.
Moreover, $(v,i)$ does not lie on a cycle of length at most $4\ell$ in $G''''$.
Hence, $Y(G'''')<Y(G''')$.
In addition, all clones on cycles of length at most $4\ell$ and their neighbours are common.
Hence, the construction can be repeated on $G''''$.
Since $Y(G''')\leq n^\gamma\ln^3n$, we ultimately obtain a $4\ell$-acyclic $G''$ with  $\dist(G',G'')\leq n^{\gamma}\ln^4n<n^{0.9}$.
\end{proof}

\begin{proof}[Proof of \Lem~\ref{Lemma_getBack}]
The assertion is immediate from \Lem s~\ref{Lemma_getBack2} and~\ref{Lemma_getBack3} and Markov's inequality.
\end{proof}

\subsection{Proof of \Lem~\ref{Lemma_countingConfigurations}}\label{Sec_countingConfigurations}
Let $\cV_\ell=\fT_\ell\cap\cV$ and for $T\in\cV_\ell$ let $n_T$ be the number of variable nodes $x$ such that $\partial^\ell[G,x]=T$.
By Stirling's formula the number
$|\Sigma(G,\ell,q,0)|$ of assignments $\sigma:V_n\to\Omega$ with marginals as prescribed by $q$ satisfies
	\begin{align}\label{eqLemma_countingConfigurations666}
	\abs{\ln|\Sigma|-\sum_{T\in\cV_\ell}n_T H(q_T)}\leq\ln^2n.
	\end{align}
Further, for $T\in\cV_\ell$ and $i\in[d_T]$ let $C_V(T,i)$ be the set of all clones $(x,i)\in C_V$ such that $t_{G,\ell}(x,i)=(T,i)$.
Moreover, let $C_F(T,i)$ be the set of all clones $(a,j)\in C_F$ such that $t_{G,\ell}(a,j)=(T,i)$.
Additionally, let $\cF_\ell(T,i)$ be the set of all pairs $(T',j)$ with $T'\in\fT_\ell\cap\cF$, $j\in[d_{T'}]$ such that
there is $(a,j)\in C_F(T,i)$ such that $\partial^{\ell+1}[G,a]=T'$.
Of course, the total number of perfect matchings between $C_V(T,i)$ and $C_F(T,i)$ equals $n_T!$.
If we fix $\sigma\in\Sigma(G,\ell,q,0)$, then any such perfect matching induces an assignment $\hat\sigma:C_F(T,i)\to\Omega$ by
mapping a clone $(a,j)\in C_F(T,i)$ matched to $(x,i)$ to the value $\sigma(x)$.
Let $B_{T,i}$ be the event that in a such random matching for all $(T',j)\in\cF_\ell(T,i)$ and all $\omega$ we have
	\begin{align*}
	\abs{\cbc{(a,j)\in C_F:\partial^{\ell+1}[G,a]=T',\hat\sigma(a,j)=\omega}}
	&=q_{T'\marg j}(\omega)
		\abs{\cbc{(a,j)\in C_F:\partial^{\ell+1}[G,a]=T'}}
	\end{align*}
Moreover, for $(T',j)\in\cF_\ell(T,i)$ let $m_{T'}$ be the number of $a\in F$ such that $\partial^{\ell+1}[G,a]=T'$.
Then
	\begin{align*}
	\pr\brk{B_t}&=\frac1{n_t!}\brk{\prod_{\omega\in\Omega}\bink{q_T(\omega)n_T}{(q_{T'\marg j}(\omega) m_{T'})_{(T',j)\in\cF_\ell(T,i)}}}
		\brk{\prod_{(T',j)\in\cF_{\ell}(T,i)}\bink{m_{T'}}{(q_{T'\marg j}(\omega)m_{T'})_{\omega\in\Omega}}}
		\prod_{(T',j)\in\cF_{\ell}(T,i),\omega\in\Omega}(q_{T'\marg j}(\omega)m_{T'})!\\
		&=\bink{n_T}{(q_T(\omega)n_T)_{\omega\in\Omega}}^{-1}\prod_{(T',j)\in\cF_\ell(T,i)}\bink{m_{T'}}{(q_{T'\marg j}(\omega)m_{T'})_{\omega\in\Omega}}
		=\exp\brk{O(\ln n)-\sum_{(T',j)\in\cF_\ell(T,i)}m_{T'}\KL{q_{T'\marg j}}{q_{T}}}.
	\end{align*}
Let $\cF_\ell=\fT_\ell\cap\cF$.
Multiplying up over all $(T,i)$, we obtain for $B=\bigcap B_{T,i}$
	\begin{align}\label{eqSillyExpectation1}
	\pr\brk{B}&=\prod_{T\in\cV_\ell}\prod_{i\in[d_T]}\pr\brk{B_{T,i}}=\exp\brk{O(\ln n)-\sum_{T'\in\cF_\ell}\sum_{j\in[d_{T'}]}m_{T'}
				\KL{q_{T'\marg j}}{q_{\partial^\ell[T'\reroot j]}}},
	\end{align}
where the constant hidden in the $O(\nix)$ depends on $\Delta,\Omega,\Psi,\Theta,\ell$ only.

Further, for $T'\in\cF_\ell$ let $S_{T'}$ be the event that for every $(\omega_1,\ldots,\omega_{d_{T'}})\in\Omega^{d_{T'}}$
we have
	$$\abs{\cbc{a\in F:\partial^{\ell+1}[G,a]=T',\forall j\in[d_{T'}]:\hat\sigma(a,j)=\omega_j}}
		=q_{T'}(\omega_1,\ldots,\omega_{d_{T'}})\abs{\cbc{a\in F:\partial^{\ell+1}[G,a]=T'}}.$$
Then
	\begin{align*}
	\pr\brk{S_{T'}|B}&=\bink{m_{T'}}{m_{T'} q_{T'}}\prod_{j\in[d_{T'}]}\bink{m_{T'}}{m_{T'} q_{T'\marg j}}^{-1}
		=\exp\brk{O(\ln n)-m_{T'}\KL{q_{T'}}{q_{T'\marg 1}\tensor\cdots\tensor q_{T'\marg d_{T'}}}}.
	\end{align*}
Hence, letting $S=\bigcap S_{T'}$, we obtain
	\begin{align}\label{eqSillyExpectation2}
	\pr\brk{S|B}&
		=\exp\brk{O(\ln n)-\sum_{T'\in\cF_\ell}m_{T'}\KL{q_{T'}}{q_{T'\marg 1}\tensor\cdots\tensor q_{T'\marg d_{T'}}}}.
	\end{align}
Once more the constant hidden in the $O(\nix)$ depends on $\Delta,\Omega,\Psi,\Theta,\ell$ only.
Further, given $S\cap B$ we have
	\begin{align}\label{eqSillyExpectation3}
	\prod_{a\in F}\psi_a(\sigma)&=\exp\brk{\sum_{T'\in\cF_\ell}m_{T'}\bck{\ln\psi_{T'}(\SIGMA)}_{q_{T'}}}.
	\end{align}
Finally, the assertion follows from (\ref{eqLemma_countingConfigurations666})--(\ref{eqSillyExpectation3}).

\subsection*{Acknowledgment}
The second author thanks Dimitris Achlioptas for inspiring discussions.

\end{document}